\documentclass[11pt,oneside,a4paper]{article}
\usepackage{amsmath, amssymb, amsthm}
\usepackage{subfig, caption, graphicx}
\usepackage{booktabs}
\usepackage{geometry}
\usepackage[bookmarks]{hyperref}
\hypersetup{
	colorlinks = true,
	linkcolor = black,
    urlcolor = black,
    citecolor = black,
	anchorcolor = black,
}
\usepackage{array}

\usepackage{tikz}
\usetikzlibrary{matrix, calc}
\usetikzlibrary{fit}

\usepackage{amsopn}

\newcommand{\dd}{\!\mathrm{d}}
\newcommand{\C}{\mathbb{C}}

\usepackage{soul}
\usepackage{xcolor}

\newtheorem{theorem}{Theorem}[section]

\newtheorem{remark}[theorem]{Remark}
\theoremstyle{definition}
\newtheorem{definition}[theorem]{Definition}
\newtheorem{example}{Example}
\newtheorem{assumption}{Assumption}
\newtheorem{glt}{GLT}

\newcommand{\email}[1]{\href{mailto:#1}{\texttt{#1}}}

\title{Blocking structures, approximation, and preconditioning\thanks{Submitted to the editors DATE.}}

\date{}
	
\author{Nikos Barakitis\thanks{Department of Science and High Technology. University of Insubria, Como, Italy. Department of Informatics. Athens University of Economics and Business, Athens, Greece. (\email{nickbar@aueb.gr})}
\and Marco Donatelli\thanks{Department of Science and High Technology. University of Insubria, Como, Italy. (\email{marco.donatelli@uninsubria.it})}
\and Samuele Ferri\thanks{Department of Science and High Technology. University of Insubria, Como, Italy. (\email{sferri1@uninsubria.it})}
\and Valerio Loi\thanks{Department of Science and High Technology. University of Insubria, Como, Italy. (\email{vloi@uninsubria.it})}
\and Stefano Serra--Capizzano\thanks{Department of Science and High Technology and CLAP Research Center. University of Insubria, Como, Italy. Department of Information Technology, Uppsala University, Uppsala, Sweden. (\email{s.serracapizzano@uninsubria.it})}
\and Rosita L. Sormani\thanks{Department of Science and High Technology. University of Insubria, Como, Italy. Department of Mathematics and Computer Science. University of Cagliari, Cagliari, Italy. (\email{rl.sormani@uninsubria.it})}
}

\begin{document}
	\maketitle
	
	\begin{abstract}
		We consider block-structured matrices $A_n$, where the blocks are of (block) unilevel Toeplitz type with $s\times t$ matrix-valued generating functions.
		Under mild assumptions on the size of the (rectangular) blocks, the asymptotic distribution of the singular values of {the} associated matrix-sequences is identified {and, when} the related singular value symbol is Hermitian, {it coincides with} the spectral symbol.
		{Building on} the theoretical derivations, we {approximate the matrices with} simplified block structures {that} show two important features: a) the related simplified matrix-sequence has the same distributions as $\{A_{{n}}\}_{{n}}$; b) a generic linear system {involving the simplified structures} can be solved in $O(n\log n)$ arithmetic operations.
		The two key properties a) and b) suggest a natural way for preconditioning a linear system with coefficient matrix $A_n$. {Under mild assumptions,} the singular value analysis and the spectral analysis of the preconditioned matrix-sequences is provided, together with {a} wide set of numerical experiments.
	\end{abstract}
	
	

	\section{Introduction}
	The approximation of systems of Ordinary or Partial Differential Equations ({ODEs} or PDEs), the reconstruction of {signals} or images with incomplete data, including the inpainting problem, all lead to block structured matrices, where the single blocks have (block) multilevel Toeplitz {structure} or, when there is no space invariance, {a more general form}; see \cite{mixed1,GLT-blocks-d-dim,GLT-blocks-1-dim,MR3689933,MR3543002,dumb} for examples in a differential context and \cite{missing-data1,blo vs imag2,missing-data2} for applications in imaging and signal processing. In fact, when dealing with {ODEs} and PDEs with variable coefficients, the resulting matrix-sequences belong to {one of the (block) multilevel Generalized Locally Toeplitz (GLT) spaces ($\ast$-algebras when the blocks are square)}, with blocks having size {that depends} on the approximation scheme and on the dimensionality $d$ of the physical domain, while the number of levels is exactly $d$: in the works \cite{GLT-blocks-d-dim,GLT-blocks-1-dim,MR3543002,GSI,GSII,tom} {and references therein,} both the theory and applications are considered. {Specifically, we mention} \cite{EMI,MR4389580,MR4623368,MR3689933,MR4284081,dumb,MR3904142} for the development of numerical algorithms based on the theory.

\subsection*{State of the art and novelty}
	Only in very recent studies \cite{pre-prequel,prequel,conj-blo-I} the spectral {and} singular value distribution in the Weyl sense have been formally derived for block structured matrix-sequences, so generalizing the GLT theories. In fact, in these initial theoretical works, the distribution analysis is  {thoroughly} described in the block unilevel setting, with rectangular blocks of different dimensions. We recall that {the} presence of zeros in the distribution function and the order of the zeros give information on the conditioning and on the size and nature of the subspaces where the ill-conditioning arise (see e.g. \cite{Se-Ti} and references therein).
	Here, in order to overcome the difficulty induced by the ill-conditioning, we consider preconditioning strategies {for} Krylov methods based on the theoretical analysis, {proving} clustering results at $1$ {for the preconditioned matrix-sequence}, both in the eigenvalue and singular value sense. Indeed, this is the first work where the distribution results for block structures are employed for computational purposes, by introducing proper preconditioners. Furthermore, the theoretical derivations on the clustering of the preconditioned matrix-sequences are well corroborated by a vast set of numerical experiments and visualizations, that include also dense and two-level block structures, with blocks approximating fractional differential operators.
	
		The current work is organized as follows. In {Section} \ref{sec:general_block} and in Section \ref{sec:Toeplitz_block} we {introduce} the problem setting and general theoretical tools, {including} the spectral and singular value analysis previously developed.
	Section \ref{sec:appr-prec} deals with the approximation of the original structures, the design of associated preconditioners, and the singular value and spectral analysis of the resulting preconditioned matrix-sequences. Section \ref{sec:numerical_tests} contains numerical tests and visualizations showing the effectiveness of the proposed preconditioners.
	{In Section \ref{sec:fin} we make some concluding remarks and discuss open problems}.
	
	\section{Problem setting and general theoretical tools}\label{sec:general_block}

	In the current section, we introduce a particular class of structured matrices. {Let $\nu$ and $n_1,n_2,\dots,n_{\nu}$ be positive integers, and $n = \sum_{i=1}^{\nu} n_i$}. We consider $A_n=[{A_{ij}}]_{i,j=1}^{\nu}$ given by
	\begin{small}
		\begin{equation} \label{eq:A_general}
			A_n = \left[
			\begin{tikzpicture}[baseline=(m.center)]
				\matrix (m) [matrix of math nodes,column sep=0.15em,row sep=0.15em] {|[draw,dashed, minimum width=0.8cm, minimum height=0.8cm]| A_{11} & |[draw,dashed, minimum width=2.4cm, minimum height=0.8cm]| A_{12} & \cdots & |[draw,dashed, minimum width=1.6cm, minimum height=0.8cm]| A_{1\nu} \\
					|[draw,dashed, minimum width=0.8cm, minimum height=2.4cm]| A_{21} & |[draw,dashed, minimum width=2.4cm, minimum height=2.4cm]| A_{22} & \cdots & |[draw,dashed, minimum width=1.6cm, minimum height=2.4cm]| A_{2\nu} \\
					\vdots & \vdots & \ddots & \vdots \\
					|[draw,dashed, minimum width=0.8cm, minimum height=1.6cm]| A_{\nu 1} & |[draw,dashed,minimum width=2.4cm, minimum height=1.6cm]| A_{\nu2} & \cdots & |[draw,dashed, minimum width=1.6cm, minimum height=1.6cm]| A_{\nu\nu} \\
				};
			\end{tikzpicture}
			\right],
		\end{equation}
	\end{small}
    {in which for $i=1,\dots,\nu$} each diagonal block {$A_{ii}$} is a matrix of size $sn_i\times tn_i$, {where $s$ and $t$ are fixed positive integers}, and the off-diagonal blocks {$A_{ij}$}, $i\neq j$, are general $sn_{i}\times tn_{j}$ rectangular matrices. {When} $s=t$, the diagonal blocks $A_{ii}$, $i=1,\ldots,\nu$, are square matrices. {Overall, $A_n$ has size $sn\times tn$ and is composed of $\nu^2$ blocks}.

	We consider a matrix-sequence $\{A_n\}_n$ {in which $A_n$ takes} the form in (\ref{eq:A_general}) {and increases in size as $n$ grows}. {Throughout the paper we work under} the following natural assumptions, {motivated by the discussion in} \cite{pre-prequel,prequel}:
    \begin{assumption} \label{assump1}
        The parameters $\nu, s,t$ are fixed independent of $n$ {and $\nu\ge 2, s,t\ge 1$} (the case $\nu=1$ is not of interest for obvious reasons);
    \end{assumption}
    \begin{assumption} \label{assump2}
        $\lim_{n,n_i\rightarrow \infty} \frac{n_i}{n}=c_i$ {with} $c_i\in (0,1)$, $i=1,\dots,\nu$, {and}  $c_1+c_2+\cdots + c_\nu=1$.
    \end{assumption}
    {Throughout this work, it is often assumed that each $c_i\in\mathbb{Q}^+$, so that $c_i = \frac{\alpha_i}{\beta_i}$ for some $\alpha_i,\beta_i\in\mathbb{N}^+$, $i=1,\dots,\nu$. Setting $m=\mathrm{lcm} (\beta_1,\ldots,\beta_\nu)$ with ${\rm lcm}$ denoting the least common multiple, we will also write $c_i = \frac{m_i}{m}$ for suitable $m_i\in\mathbb{N}^+$, $i=1,\dots,\nu$.  The assumption on the rational character of the ratios $c_i$, $i=1,\dots,\nu$, is essential in the proof techniques in \cite{pre-prequel,prequel}, while the difficulty has been circumvented in \cite{conj-blo-I}, so allowing irrational ratios, by using the new notion of generalized approximating class of sequences (see \cite{gacs} for the new notion and related concepts and \cite{conj-blo-I} for a quite involved application of the new tools).
When dealing with coupled PDEs as the basic elasticity problem \cite{braess} or the more involved EMI model \cite{EMI}, the involved ratios are often rational, but also the irrational setting is of interest, when considering e.g. a circular domain embedded in the square $[0,1]^2$ and mixed discretizations \cite{mixed1}.

    }

	We report the notion of Weyl distributions in the spectral and singular value sense {associated to a matrix-sequence}. We choose a very general formulation, {as needed by} various concrete examples. {Whenever measurability is mentioned, we refer to the Lebesgue measure; in particular $m_\ell$ denotes the Lebesgue measure in $\mathbb{R}^\ell$, with $\ell$ a positive integer. Moreover, when working with matrix-valued functions, operations and properties such as continuity are intended componentwise}.
	
	\begin{definition}[\cite{GLT-blocks-d-dim,GSI,GSII,MR0890515,TyZ}]
		Let ${f}:\Omega \to\mathbb{C}^{s\times t}$ be a  measurable function defined on a measurable set $\Omega\subset\mathbb R^\ell$ with $\ell\ge 1$ {and}
		$0<m_\ell(\Omega)<\infty$. Let $\mathcal{C}_0(\mathbb{K})$ be the set of continuous functions with compact support over $\mathbb{K}\in \{\mathbb{C}, \mathbb{R}_0^+\}$ and let $\{A_{{n}}\}_{{n}}$ be a sequence of matrices, {where $A_n$ has size $sn\times tn$ and singular values} $\sigma_j(A_{{n}})$, $j=1,\ldots,{rn}$, {with $r=\min\{s,t\}$}.
		\begin{itemize}
			\item $\{A_{{n}}\}_{n}$ is distributed as ${f}$ in the sense of the singular values, and we write
			\begin{align*}
				\{A_{{n}}\}_{{n}}\sim_\sigma{f},\nonumber
			\end{align*}
			if the following limit relation holds for {any} $F\in\mathcal{C}_0(\mathbb{R}_0^+)$
			\begin{align}\label{eq:distribution_sv}
				\lim_{n\to\infty}\frac{1}{{rn}}\sum_{j=1}^{{rn}} F(\sigma_j(A_{n}))=\frac1{m_\ell(\Omega)}\int_{\Omega}\frac{\sum_{i=1}^{r}F\left(\sigma_i\left({f}\left(\boldsymbol{\theta}\right)\right)\right)}{r}\,\dd{\boldsymbol{\theta}}.
			\end{align}
			
			The function ${f}$ is called the {singular value symbol} {and it} describes the singular value distribution of $ \{A_{{n}}\}_{{n}}$.
			
			\item In the case where {$r=s=t$} and {each} $A_n$ is square of size {$rn$} with eigenvalues $\lambda_j(A_n)$, $j=1,\ldots,{rn}$, the matrix-sequence  $\{A_n\}_n$ is distributed as $f$ in the sense of the eigenvalues, and we write
			\begin{align*}
				\{A_{{n}}\}_{{n}}\sim_\lambda{f},\nonumber
			\end{align*}
			if the following limit relation holds for all $F\in\mathcal{C}_0(\mathbb{C})$
			\begin{align}\label{eq:distribution_eig}
				\lim_{{n}\to\infty}\frac{1}{{{rn}}}\sum_{j=1}^{{{rn}}}F(\lambda_j(A_{n}))=\frac1{m_\ell(\Omega)}\int_{\Omega}  \frac{\sum_{i=1}^{{r}} F\left(\lambda_i\left({f} \left(\boldsymbol{\theta}\right)\right)\right)}{{r}}\,\dd{\boldsymbol{\theta}}.
			\end{align}
			
			The function ${f} $ is called the eigenvalue symbol and it describes the eigenvalue distribution of $\{A_{{n}}\}_{{n}}$.
		\end{itemize}
	\end{definition}

    {The notion of clustering can be seen as a special case of distribution.
\begin{definition} \label{def:cluster}
    Given $p\in\mathbb{R}^+_0$, a matrix-sequence $\{A_n\}_n$, with $A_n\in\mathbb{C}^{sn \times tn}$, is (weakly) clustered at $p$ in the sense of the singular values if, for any $\epsilon>0$
        \begin{equation*}
            \#\big\{ j\in\{1,\ldots,rn\} : \lvert\sigma_j(A_n)-p\rvert > \epsilon \big\} = o(n), \qquad\text{as } n\to\infty,
        \end{equation*}
    where $r=\min\{s,t\}$. Moreover, $\{A_n\}_n$ is strongly clustered at $p$ in the sense of the singular values if for any $\epsilon>0$,
        \begin{equation*}
            \#\big\{ j\in\{1,\ldots,rn\} : \lvert\sigma_j(A_n)-p\rvert > \epsilon \big\} = O(1), \qquad\text{as } n\to\infty.
        \end{equation*}
    When $A_n$ is square, corresponding definitions can be given for the eigenvalues, considering $p\in\mathbb{C}$ and a neighborhood $B(p,\epsilon)$ of $p$ with radius $\epsilon$.
\end{definition}
A weak cluster at $p$ means that, as $n$ tends to infinity, the number of singular values or eigenvalues lying farther than $\epsilon$ from $p$ is negligible with respect to the size of the matrix, while in a strong cluster it is bounded by a constant independent of $n$. Hence, when considering Definition \ref{def:cluster}, the distributions of the matrix-sequence can be described by a constant function $p$.}

{For our Krylov preconditioning purposes, the case of distribution and clustering at $1$ retains special significance. Indeed, in that setting, the ideal scenario is clearly the strong clustering, in which the number of outliers is bounded by a constant uniformly independent of $n$.}

The last tool we introduce falls in the category of the extradimensional results \cite{pre-prequel,prequel,MR3904142}, whose initial idea goes back to an unpublished  work by Tyrtyshnikov (see the discussion in \cite{pre-prequel,MR3904142}).

	\begin{theorem}[\cite{pre-prequel,prequel}]\label{Extradimensional}
		Let $\{X_n\}_n$ be a given (rectangular) matrix-sequence with $X_n$ of order $n^{(1)} \times n$, $n^{(1)} \geq n$. Let $P_n \in \mathbb{C}^{n \times n'}$, $P_{n^{(1)}} \in \mathbb{C}^{n^{(1)} \times n^{(1)'}}$ be two compression matrices with $n' < n$, ${n^{(1)'}} < n^{(1)}$ and $P_n^*P_n = I_{n'}$, $P_{n^{(1)}}^*P_{n^{(1)}} = I_{n^{(1)'}}$, and let us consider $Y_{n'} = P_{n^{(1)}}^* X_n P_n$. Under the assumption that
		\begin{displaymath}
			\lim_{n,n' \to \infty} \frac{n'}{n} = \lim_{{n^{(1)}},{n^{(1)'}} \to \infty} \frac{{n^{(1)'}}}{{n^{(1)}}} = 1,
		\end{displaymath}
		i.e. $n = n' + o(n)$, ${n^{(1)}} = {n^{(1)'}} + o({n^{(1)}})$, we have
		\begin{displaymath}
			\{X_n\}_n \sim_\sigma f \quad \text{iff} \quad \{Y_{n'}\}_{n'} \sim_\sigma f.
		\end{displaymath}
	\end{theorem}

	\subsection{Basic structured matrix-sequences and (rectangular) GLT sequences}
	
	In this section, we {briefly review} the fundamental structured matrix-sequences {relevant to} our problem and {essential for obtaining} our distribution and approximation results. The {structures} under consideration are Toeplitz and Hankel sequences {generated by functions in} $L^1([-\pi,\pi], {\C^{s \times t}})$ and zero-distributed sequences. {Such} sequences are part of the GLT {spaces/$\ast$}-algebras, {the theoretical framework in which we are able to derive our results. Let us} provide a {concise} overview of these tools. \\
	
	We begin by discussing zero-distributed sequences, which {form} a significant and interesting {$\ast$}-algebra in their own right in the square case {and} play a crucial role in explaining the clustering {behavior} of the preconditioners.
	\begin{definition}\label{def:zero-distributed-sequence}
		A {matrix-sequence $\{Z_n\}_n$, with $Z_n \in \C^{sn \times tn}$,} is zero-distributed if $\{Z_n\}_n \sim_\sigma 0$, i.e.,
		\begin{displaymath}
			\lim_{n \to \infty} \frac{1}{{rn}} \sum_{i=1}^{{rn}} F(\sigma_i(Z_n)) = F(0), \quad \forall F \in \mathcal{C}_c(\mathbb{R}),
		\end{displaymath}
		{where $r=\min\{s,t\}$}.
	\end{definition}
	This is equivalent to saying that the singular values are (weakly) clustered at $0$, {according to Definition \ref{def:cluster}}. {A useful characterization of zero-distributed sequences is given by the following theorems; see \cite{taud2}.}
	
	\begin{theorem}\label{thm:zero-characterization}
		Let $\{Z_n\}_n$ be a matrix-sequence, with $Z_n \in \C^{sn \times tn}$. The following conditions are equivalent.
		\begin{enumerate}
			\item $\{Z_n\}_n$ is a zero-distributed sequence.
			\item For every $n$ {it holds} $Z_n = R_n + N_n$, where, {denoting with $\|\cdot\|$ the spectral norm},
			\begin{displaymath}
				\lim_{n \to \infty} \frac{\operatorname{rank}(R_n)}{n} = \lim_{n \to \infty} \|N_n\| = 0.
			\end{displaymath}
		\end{enumerate}
	\end{theorem}

\begin{theorem}\label{thm:zero-char2}
		Let $\{Z_n\}_n$ be a matrix-sequence, with $Z_n \in \C^{sn \times tn}$ and let $\|\cdot\|_p$ the Schatten $p$-norm, i.e., the $p$ norm of the vector of the singular values of the argument with $p\ge 1$; see {\rm \cite{Bhatia-book}}. If $\|Z_n\|_p^p=o(n)$ then $\{Z_n\}_n$ is a zero-distributed sequence.
	\end{theorem}

	We now describe {Toeplitz matrix-sequences, which constitute the main ingredient} in the considered problem. {More specifically, we consider Toeplitz matrices constructed from} a univariate matrix-valued function $f$, {called the} \textit{generating function}. Refer to \cite{BS,MR0890515} for the notion in the scalar-valued case and \cite{GLT-blocks-1-dim} for the matrix-valued setting.
	
	\begin{definition}\label{def:toeplitz}
		Let $f$ be a $s\times t$ matrix-valued function belonging to $L^1([-\pi,\pi],{\C^{s\times t}})$ and periodically extended to the whole real line. {The} Toeplitz matrix $T_{n}(f) \in \mathbb{C}^{sn \times tn}$ associated with $f$ is defined as
		\begin{equation}\label{eq:toe_def-d=1}
			T_n(f)=\left[\hat f_{i-j}\right]_{i,j=1}^n,\nonumber
		\end{equation}
		where
		\begin{equation}\label{fourier-toeplitz}
			\hat{f}_{k}=\frac1{2\pi}\int_{-\pi}^{\pi}\!\!f(\theta)\,{\rm e}^{-k i \theta}\dd\theta \in \mathbb{C}^{s \times t},\qquad k\in\mathbb Z,\qquad \iota^2=-1,
		\end{equation}
		are the Fourier coefficients of $f$, {in which the integral is computed componentwise}. For $s=t$, $T_{n}(f)$ is square.
	\end{definition}
\begin{definition}\label{def:rectangular_toeplitz}
		{Extending} Definition \ref{def:toeplitz}, for $n'\neq n$ it is possible to define a (inherently) rectangular $sn\times tn'$  matrix $T_{n,n'}(f)$ as
		\begin{equation}\label{eq:rectangular_Toeplitz}
			T_{n,n'}(f)= \left[\hat f_{i-j}\right]_{\stackrel{\substack{i=1,\ldots,n \\ j=1,\ldots,n'}}{}}.
		\end{equation}
	\end{definition}

    Similarly to the Toeplitz case, we can construct Hankel matrices from the Fourier coefficients of a function \cite{MR1740439}.

	\begin{definition}\label{def:hankel}
		Let $f$ be a $s\times t$ matrix-valued function belonging to $L^1([-\pi,\pi],{\C^{s\times t}})$ and periodically extended to the whole real line.
    The Hankel matrix of size $sn\times tn$ generated by $f\in L^1([-\pi,\pi],\C^{s\times t})$ is defined as
	\begin{equation}\label{eq:hankel_def-d=1}
		H_n(f)= \sum_{k=1}^{2n-1} K_{n}^{(k)} \otimes  \hat{{f}}_{k},
	\end{equation}
	where $K_{n}^{(k)}$ denotes the matrix of order $n$ whose $(i, j )$ entry equals $1$ if $j + i = k + 1$ and zero otherwise, while $\hat{{f}}_{k}$, $k \in\mathbb Z$, are the Fourier coefficients of $f$ as in equation (\ref{fourier-toeplitz}).
  \end{definition}

For these structures, many norm and distribution results are known.

    	\begin{theorem}[\cite{Se-Ti-LPO}] \label{thm:bound_toeplitz}
		If \(n \in \mathbb{N}, 1 \leq p \leq \infty\) and \(f \in L^p([-\pi, \pi], {\C^{s \times t}})\), then
		\begin{equation}
			\|T_n(f)\|_p \leq n^{1/p} \frac{\|f\|_{L^p}} {(2\pi)^{1/p}}.
		\end{equation}
	\end{theorem}
    {
    \begin{theorem}[\cite{MR1671591}]\label{thm:szego-d=1}
	   Let {${f}\in L^1([-\pi,\pi],\C^{s\times t})$}. Then,
		\begin{equation}
			\label{eq:szego-sv-d=1}
			\{T_{n}({f})\}_{{n}}\sim_\sigma~{f}.
		\end{equation}
		In addition, if the generating function $f$ is a Hermitian-valued function, then necessarily $s=t$ and 	
		\begin{equation}
			\label{eq:szego-eig-d=1}
			\{T_{n}({f})\}_{{ n}}\sim_\lambda~{f}.
		\end{equation}
	\end{theorem}
    }
	
	\begin{theorem}[\cite{MR1740439}]\label{thm:fasino-tilli-d=1}
		Let {$f\in L^1([-\pi,\pi],\C^{s\times t})$}. Then,
		\begin{equation}
			\label{eq:fasino-tilli-d=1}
			\{H_{n}({f})\}_{{n}}\sim_\sigma~0.
		\end{equation}
	\end{theorem}
	
	\begin{theorem}[\cite{MR1740439}]\label{thm:bound_hankel}
		If \(n \in \mathbb{N}, 1 \leq p \leq \infty\) and \(f \in L^p([-\pi, \pi], {\C^{s \times t}})\), then
		\begin{equation}
			\|H_n(f)\|_p \leq n^{1/p} \frac{\|f\|_{L^p}} { (2\pi)^{1/p}}.
		\end{equation}
	\end{theorem}
	
	\begin{remark}\label{hankel_continuous}
		{Theorem} \ref{thm:bound_hankel} is also saying that, if $f$ is continuous, the zero clustering described by {Theorem} \ref{thm:fasino-tilli-d=1} is strong.
		In fact, let $f \in C([-\pi,\pi], {\C^{s \times t}})$: if we fix a small $\epsilon > 0$ and consider a trigonometric polynomial $p_\epsilon$, with a certain degree $m_\epsilon$, such that $\| f - p_\epsilon \|_{L^\infty} < \epsilon$, then by linearity of the Hankel operator
		\begin{equation}
			H_n(f)=H_n(p_\epsilon)+H_n(f-p_\epsilon).
		\end{equation}
		{Theorem \ref{thm:bound_hankel} implies} that $\|H_n(f-p_\epsilon)\| \le \epsilon$, and by construction $H_n(p_\epsilon)$ has a rank bounded by a constant $r_\epsilon$ determined by $\epsilon$, but independent of the size of $H_n(f)$. More precisely, $r_\epsilon$ is bounded by the degree of $p_\epsilon$ and this ensures the strong cluster of {$\{H_n(f)\}_n$ at} $0$.
	\end{remark}	
	
	{Next, we briefly outline the main operative properties of the GLT spaces. The full constructive definition can be found in \cite{GLT-blocks-1-dim} and \cite{MR4449208}; for our purposes it is enough to know that} a unilevel $(s,t)$-block GLT sequence $\{A_n\}_n$ is a {$s\times t$}-block matrix-sequence equipped with {an essentially unique} measurable function $\kappa : [0,1] \times [-\pi,\pi] \to \mathbb{C}^{s \times t}$, called GLT symbol, {and we write} $\{A_n\}_n \sim_{\mathrm{GLT}} \kappa$. {Below} we list the {main working} properties.

    \begin{glt} \label{glt1-distributions}
		If $\{A_n\}_n \sim_{\mathrm{GLT}} \kappa$, then $\{A_n\}_n \sim_{\sigma} \kappa$. If moreover each $A_n$ is Hermitian, then $\{A_n\}_n \sim_{\lambda} \kappa$.
    \end{glt}
	\begin{glt} \label{glt2-buildingblocks}
	    The following properties hold.
        \begin{itemize}
		  \item $\{T_n(f)\}_n \sim_{\mathrm{GLT}} \kappa(x,\theta) = f(\theta)$, with $f\in L^1([-\pi,\pi], \C^{s \times t})$.
		  \item $\{Z_n\}_n \sim_{\mathrm{GLT}} \kappa(x,\theta) = 0$ if and only if $\{Z_n\}_n \sim_{\sigma} 0$.
        \end{itemize}
	\end{glt}
    \begin{glt} \label{glt3-algebra}
        Suppose that $\{A_n\}_n \sim_{\mathrm{GLT}} \kappa$ and $\{B_n\}_n \sim_{\mathrm{GLT}} \xi$. Then the following properties hold.
    	\begin{enumerate}
    		\item $\{A_n^*\}_n \sim_{\mathrm{GLT}} \kappa^*$, {where $\cdot^*$ denotes the conjugate transpose}.
    		\item $\{\alpha A_n + \beta B_n\}_n \sim_{\mathrm{GLT}} \alpha \kappa + \beta \xi$ for all $\alpha, \beta \in \mathbb{C}$, if $\kappa+\xi$ is defined.
    		\item $\{A_n B_n\}_n \sim_{\mathrm{GLT}} \kappa \xi$, if $\kappa \xi$ is defined.
    		\item If all the singular values of $\kappa$ are non-zero a.e., then $\{A_n^{\dagger}\}_n \sim_{\mathrm{GLT}} \kappa^{\dagger}$, {where $\cdot^\dagger$ denotes the Moore-Penrose inverse}.
    		\item If $\{A_n\}_n \sim_{\mathrm{GLT}} \kappa$ and each $A_n$ is Hermitian, then $\{f(A_n)\}_n \sim_{\mathrm{GLT}} f(\kappa)$ for every continuous function $f : \mathbb{C} \to \mathbb{C}$.
    	\end{enumerate}
    \end{glt}
To be more precise, in axiom \textbf{GLT \ref{glt3-algebra}}, item 1 and item 4, $\{A_n\}_n$ is unilevel $(s,t)$-block GLT sequence and $\{A_n^*\}_n$ and $\{A_n^{\dagger}\}_n$ are both unilevel $(t,s)$-block GLT sequences. Similarly, in axiom \textbf{GLT \ref{glt3-algebra}}, item 3, $\{A_n\}_n$ unilevel $(s,t)$-block GLT sequence and $\{B_n\}_n$
unilevel $(t,v)$-block GLT sequence imply $\{A_nB_n\}_n$ unilevel $(s,v)$-block GLT sequence. Furthermore, when $s=t$ the corresponding GLT space is a GLT $\ast$-algebra.

	\section{Structured matrices with block unilevel Toeplitz blocks}\label{sec:Toeplitz_block}
	
	{In this section we consider} matrices of the form (\ref{eq:A_general}) in which
	\begin{itemize}
		\item the diagonal blocks $A_{ii}$, {for $i=1,\dots, \nu$,} are $sn_i\times tn_i$ Toeplitz matrices {$T_{n_i}(f_{i,i})$}, {following Definition \ref{def:toeplitz}};
		\item the off-diagonal blocks $A_{ij}$, {for $i\neq j$}, are $sn_{i}\times tn_{j}$ rectangular matrices $T_{n_i,n_j}(f_{i,j})$, according to {Definition~\ref{def:rectangular_toeplitz}};
	\end{itemize}
    where $f_{i,j}$, $i,j=1,\dots,\nu$, are {given} functions {belonging to $L^1([-\pi,\pi],{\C^{s\times t}})$}. More explicitly, the matrix $A_n$ has the form
	\begin{small}
		\begin{equation} \label{eq:A_Toeplitz}
			A_n=\left[
			\begin{tikzpicture}[baseline=(m.center)]
				\matrix (m) [matrix of math nodes,column sep=0.15em,row sep=0.15em] {
					|[draw,dashed, minimum width=2cm, minimum height=2cm]| T_{n_1}(f_{1,1}) & |[draw,dashed, minimum width=2.4cm, minimum height=2cm]| T_{n_1,n_2}(f_{1,2}) & \cdots & |[draw,dashed, minimum width=1.6cm, minimum height=2cm]| T_{n_1,n_{\nu}}(f_{1,{\nu}}) \\
					|[draw,dashed, minimum width=2cm, minimum height=2.4cm]| T_{n_2,n_1}(f_{2,1}) & |[draw,dashed, minimum width=2.4cm, minimum height=2.4cm]| T_{n_2}(f_{2,2}) & \cdots & |[draw,dashed, minimum width=1.6cm, minimum height=2.4cm]| T_{n_2,n_{\nu}}(f_{2,{\nu}}) \\
					\vdots & \vdots & \ddots & \vdots \\
					|[draw,dashed, minimum width=2cm, minimum height=1.6cm]| T_{n_{\nu},n_1}(f_{{\nu},1}) & |[draw,dashed,minimum width=2.4cm, minimum height=1.6cm]| T_{n_{\nu},n_2}(f_{{\nu},2}) & \cdots & |[draw,dashed, minimum width=1.6cm, minimum height=1.6cm]| T_{n_{\nu}}(f_{\nu,\nu}) \\
				};
			\end{tikzpicture}
			\right].
		\end{equation}
	\end{small}

	{We report some useful results on this type of structures, proven and discussed in \cite{prequel}}.
	
	\begin{theorem}[\cite{prequel}]\label{equal-blocks-basic}
		Assume that {a matrix} $A_n$ {structured} as in (\ref{eq:A_Toeplitz}) is given, with $n_1=n_2=\cdots=n_\nu=\frac{n}{\nu}$, {and set} $n(\nu)=\frac{n}{\nu}$. For any positive integer $\mu$ define the following permutation matrix
		\begin{displaymath}
			\Pi_\mu=\Pi_{n,\nu,\mu} = \begin{bmatrix}
				I_\nu \otimes \mathbf{e}_{1}^T \\
				I_\nu \otimes \mathbf{e}_{2}^T \\
				\vdots \\
				I_\nu \otimes \mathbf{e}_{n(\nu)}^T
			\end{bmatrix} \otimes I_\mu,
		\end{displaymath}
such that
		\begin{displaymath}
			T_n(F)=\Pi_s A_n \Pi_t^T,
		\end{displaymath}
		where $F=\left(f_{i,j}\right)_{i,j=1}^\nu$. With the previous assumptions we have
		\begin{displaymath}
			\{A_{n}\}_{{n}}\sim_\sigma~{F}.
		\end{displaymath}
		{If in addition $F$ is Hermitian-valued (which implies $s=t$)}, then
		\begin{displaymath}
			\{A_{n}\}_{{n}}\sim_\lambda~{F}.
		\end{displaymath}
	\end{theorem}

	\begin{theorem}[\cite{prequel}]\label{almost equal-blocks-basic}
		Assume that $A_n$ is {a matrix structured} as in (\ref{eq:A_Toeplitz}), with $n_j=n(\nu)+o(n)$, $j=1,\ldots,\nu$, {and set} $n(\nu)=\frac{n}{\nu}$,  $F=\left(f_{i,j}\right)_{i,j=1}^\nu$. With the previous assumptions we have
		\[
		\{A_{n}\}_{{n}}\sim_\sigma~{F}.
		\]
		{If in addition $F$ is Hermitian-valued (which implies $s=t$), then}
		\[
		\{A_{n}\}_{{n}}\sim_\lambda~{F}.
		\]
	\end{theorem}

    \begin{theorem}[\cite{prequel}]\label{th:general}
		Suppose that $A_n$ is {a matrix structured} as in (\ref{eq:A_Toeplitz}) with $\nu,n_1,\ldots,n_\nu$ satisfying {Assumptions \ref{assump1} and \ref{assump2}, in which we assume $c_j\in\mathbb{Q}^+$, so that $c_j=\frac{\alpha_j}{\beta_j}$ with $\alpha_j,\beta_j\in\mathbb{N}^+$ for $j=1,\ldots,\nu$. Let
		$F=\left(E_{i,j}\right)_{i,j=1}^\nu$, where
		\[
		E_{j,k} =
		\begin{cases}\label{block_symbols}
			\begin{array}{l l l l}
				I_{m_j}\otimes f_{j,j} & & \text{if } j=k, \\
				& & \\
				\left[
				\begin{array}{c|c}
					I_{m_j}\otimes f_{j,k}& O_{sm_j\times t(m_k-m_j)} \\
				\end{array}
				\right] &  &\text{if } j\neq k, & \text{and } m_k\ge m_j, \\
				& & \\
				\left[
				\begin{array}{c}
					I_{m_k}\otimes f_{j,k}\\
					\hline
					O_{s(m_j-m_k)\times tm_k}
				\end{array} \right]& & \text{if } j\neq k, & \text{and } m_j\ge m_k, \\
			\end{array}
		\end{cases}
		\]
        with $c_j=\frac{m_j}{m}$, $m={\rm lcm}(\beta_1,\ldots,\beta_\nu)$ and $m_j\in\mathbb{N}^+$ computed accordingly}. Then, we have
		\begin{displaymath}
			\{A_{n}\}_{{n}}\sim_\sigma~{F}.
		\end{displaymath}
		If in addition $A_n$ is Hermitian, then
		\begin{displaymath}
			\{A_{n}\}_{{n}}\sim_\lambda~{F}.
		\end{displaymath}
	\end{theorem}

	{The latter theorem is proved in \cite{prequel} by introducing two} auxiliary matrices $\widetilde{A}_n$, $\widehat{A}_n$ {designed to neglect terms in $A_n$ that lead to zero-distributed sequences. Following a similar approach,} we set
	\begin{small}
		\begin{equation} \label{eq:tilde A_Toeplitz}
			\widetilde A_n=\left[
			\begin{tikzpicture}[baseline=(m.center)]
				\matrix (m) [matrix of math nodes,column sep=0.15em,row sep=0.15em] {
					|[draw,dashed, minimum width=2cm, minimum height=2cm]| T_{n_1}(f_{1,1}) & |[draw,dashed, minimum width=2.4cm, minimum height=2cm]| \widetilde T_{n_1,n_2} & \cdots & |[draw,dashed, minimum width=1.6cm, minimum height=2cm]| \widetilde T_{n_1,n_{\nu}} \\
					|[draw,dashed, minimum width=2cm, minimum height=2.4cm]| \widetilde T_{n_2,n_1} & |[draw,dashed, minimum width=2.4cm, minimum height=2.4cm]| T_{n_2}(f_{2,2}) & \cdots & |[draw,dashed, minimum width=1.6cm, minimum height=2.4cm]| \widetilde T_{n_2,n_{\nu}} \\
					\vdots & \vdots & \ddots & \vdots \\
					|[draw,dashed, minimum width=2cm, minimum height=1.6cm]| \widetilde T_{n_{\nu},n_1} & |[draw,dashed,minimum width=2.4cm, minimum height=1.6cm]| \widetilde T_{n_{\nu},n_2} & \cdots & |[draw,dashed, minimum width=1.6cm, minimum height=1.6cm]| T_{n_{\nu}}(f_{\nu,\nu}) \\
				};
			\end{tikzpicture}
			\right],
		\end{equation}
	\end{small}
{where the off-diagonal blocks are defined as
    \begin{alignat*}{2}
        \widetilde T_{n_i,n_j} &= \left[\begin{array}{c}
			T_{n_j}(f_{i,j}) \\[2pt]
			\hline
			O_{s(n_i-n_j)\times tn_j}
		\end{array} \right]
        && \quad\text{if } n_i > n_j, \\[5pt]
        \widetilde T_{n_i,n_j} &= \Big[
		\begin{array}{c|c}
			T_{n_i}(f_{i,j}) & O_{sn_i\times t(n_j-n_i)} \\
		\end{array}
		\Big]
        && \quad\text{if } n_i < n_j, \\[5pt]
        \widetilde T_{n_i,n_j} &= T_{n_i}(f_{i,j})
		        && \quad\text{if } n_i = n_j.
    \end{alignat*}
Now, we assume that the sizes $n_1,\ldots,n_\nu$ can be written as ${n_j=m_j n_{m}+o(n_{m})}$, with $m$ a fixed positive integer, and define}
	\begin{small}
		\begin{equation} \label{eq:hat A_Toeplitz}
			\widehat A_n=\left[
			\begin{tikzpicture}[baseline=(m.center)]
				\matrix (m) [matrix of math nodes,column sep=0.15em,row sep=0.15em] {
					|[draw,dashed, minimum width=2cm, minimum height=2cm]| \widehat T_{n_1} & |[draw,dashed, minimum width=2.4cm, minimum height=2cm]| \widehat T_{n_1,n_2} & \cdots & |[draw,dashed, minimum width=1.6cm, minimum height=2cm]| \widehat T_{n_1,n_{\nu}} \\
					|[draw,dashed, minimum width=2cm, minimum height=2.4cm]| \widehat T_{n_2,n_1} & |[draw,dashed, minimum width=2.4cm, minimum height=2.4cm]| \widehat T_{n_2} & \cdots & |[draw,dashed, minimum width=1.6cm, minimum height=2.4cm]| \widehat T_{n_2,n_{\nu}} \\
					\vdots & \vdots & \ddots & \vdots \\
					|[draw,dashed, minimum width=2cm, minimum height=1.6cm]| \widehat T_{n_{\nu},n_1} & |[draw,dashed,minimum width=2.4cm, minimum height=1.6cm]| \widehat T_{n_{\nu},n_2} & \cdots & |[draw,dashed, minimum width=1.6cm, minimum height=1.6cm]| \widehat T_{n_{\nu}} \\
				};
			\end{tikzpicture}
			\right],
		\end{equation}
	\end{small}
{which in turn is based on the structure of $\widetilde A_n$. { Essentially, considering the block $(j,k)$ in (\ref{eq:tilde A_Toeplitz}) with $n_j \le n_k$, the matrix $T_{n_{j}}(f)$ in the considered block is essentially replaced by $m_j$ copies of $T_{n_{m}}(f)$, that is,
\[
\left[I_{m_j-1} \otimes T_{n_{m}}(f)\right] \oplus X_{n_{m},j,k}
\]
and $X_{n_{m},j,k}$ is a matrix of size $n_m + o(n_m)$, obtained by neglecting, if the sign of the $o(n_m)$ term is negative, or adding, if the sign of the $o(n_m)$ term is positive, the last $o(n_m)$ rows and columns in $T_{n_{m}}(f)$. Similarly, considering the block $(j,k)$ in (\ref{eq:tilde A_Toeplitz}) with $n_k \le n_j$, the matrix $T_{n_{k}}(f)$ in the considered block is essentially replaced by $m_k$ copies of $T_{n_{m}}(f)$, that is,
\[
\left[I_{m_k-1} \otimes T_{n_{m}}(f)\right] \oplus X_{n_{m},j,k}
\]
and $X_{n_{m},j,k}$ is a matrix of size $n_m + o(n_m)$, obtained by neglecting, if the sign of the $o(n_m)$ term is negative, or adding, if the sign of the $o(n_m)$ term is positive, the last $o(n_m)$ rows and columns $o(n_m)$ in $T_{n_{m}}(f)$.
}}\\
In \cite{prequel}, it is proven that
	\begin{align}\label{1st zero-distr}
		\{A_{n}-\widetilde A_n\}_n&\sim_\sigma 0, \\ \label{2nd zero-distr}
		\{\widetilde A_{n}-\widehat A_n\}_n&\sim_\sigma 0.
	\end{align}
	{In particular,} $\{A_{n}-\widetilde A_n\}_{{n}}$ {and $\{\widetilde A_{n}-\widehat A_n\}_{{n}}$} can be expressed {respectively} as the sum of $\nu^2-\nu$ {and $\nu(m^2-m)$} zero-distributed sequences, with {$m$ as given in Theorem \ref{th:general}}.





	\section{Approximation and preconditioning}\label{sec:appr-prec}
	
    {Here we propose a strategy to construct a preconditioner $S_n$ for a structure $A_n$ as in \eqref{eq:A_Toeplitz}. We start from \( \widehat{A}_n \) defined in \eqref{eq:hat A_Toeplitz} and replace some terms with suitable approximations:} each Toeplitz block {$T_{n_{m}}(f_{i,j})$} is substituted with a matrix {\( S_{n_{m}}(f_{i,j}) \)} and each diagonal remainder  $X_{n_{m},j,k}$ with some { $R_{n_{m},j,k}$}, {so that zero-distributed sequences are obtained by subtraction}. To be more precise, {consider matrices $S_{n_{m}}(f_{i,j})$ and $R_{n_{m},j,k}$ chosen so that they satisfy the conditions
    \begin{align*}
        \{S_{n_{m}}(f_{i,j}) - T_{n_{m}}(f_{i,j})\}_{n_{m}} &\sim_{\sigma}~{0}, \\
         \{X_{n_{m},j,k} - R_{n_{m},j,k}\}_{n_{m}} &\sim_\sigma~0.
    \end{align*}
    Assuming once again that ${n_j = m_j n_{m} + o(n_j)}$, with $m$ a fixed positive integer, we define
	\begin{small}
		\begin{equation} \label{eq:S_n}
			{S_n}=\left[
			\begin{tikzpicture}[baseline=(m.center)]
				\matrix (m) [matrix of math nodes,column sep=0.15em,row sep=0.15em] {
					|[draw,dashed, minimum width=2cm, minimum height=2cm]|  {\widehat S_{n_1}} & |[draw,dashed, minimum width=2.4cm, minimum height=2cm]|  {\widehat S_{n_1,n_2}} & \cdots & |[draw,dashed, minimum width=1.6cm, minimum height=2cm]| {\widehat S_{n_1,n_{\nu}}} \\
					|[draw,dashed, minimum width=2cm, minimum height=2.4cm]| {\widehat S_{n_2,n_1}} & |[draw,dashed, minimum width=2.4cm, minimum height=2.4cm]| {\widehat S_{n_2}} & \cdots & |[draw,dashed, minimum width=1.6cm, minimum height=2.4cm]| {\widehat S_{n_2,n_{\nu}}} \\
					\vdots & \vdots & \ddots & \vdots \\
					|[draw,dashed, minimum width=2cm, minimum height=1.6cm]| {\widehat S_{n_{\nu},n_1}} & |[draw,dashed,minimum width=2.4cm, minimum height=1.6cm]| {\widehat S_{n_{\nu},n_2}} & \cdots & |[draw,dashed, minimum width=1.6cm, minimum height=1.6cm]| {\widehat S_{n_{\nu}}} \\
				};
			\end{tikzpicture}
			\right],
		\end{equation}
	\end{small}
    
    in which the blocks are given by $\widehat S_{n_i} = S_{n_i}(f_{i,i})$ and
    \begin{alignat*}{2}
        \widehat S_{n_i,n_j} &= \left[ \begin{array}{c}
			S_{n_j}(f_{i,j}) \\[2pt]
			\hline
			O_{s(n_i-n_j)\times tn_j}
		\end{array} \right]
        && \quad\text{if } n_i > n_j, \\[5pt]
        \widehat S_{n_i,n_j} &= \Big[
		\begin{array}{c|c}
			S_{n_i}(f_{i,j}) & O_{sn\times t(n_j-n_i)} \\
		\end{array}
		\Big]
        && \quad\text{if } n_i < n_j, \\[5pt]
        \widehat S_{n_i,n_j} &=
			S_{n_i}(f_{i,j}) && \quad\text{if } n_i = n_j,
    \end{alignat*}
    where, if $n_j \le n_k$
	\begin{equation}\label{approximation_S1}
		{S_{n_j(f_{i,j})} = \left[I_{m_j-1} \otimes S_{n_{m}}(f_{i,j})\right] \oplus R_{n_{m},j,k}},
	\end{equation}
    and, if $n_k \le n_j$,
    \begin{equation}\label{approximation_S2}
		{S_{n_j(f_{i,j})} = \left[I_{m_k-1} \otimes S_{n_{m}}(f_{i,j})\right] \oplus R_{n_{m},j,k}},
	\end{equation}
and ${R_{n_{m},j,k}}$ is a matrix of size ${n_m + o(n_m)}$ obtained by neglecting or adding ${o(n_m)}$ rows and columns in $S_{n_{m}}(f_{i,j})$.

    To demonstrate the approximation quality of \( S_n \), we aim to prove that
	\begin{equation}\label{0-distr}
		\{S_n - A_n\}_n \sim_\sigma~0,
	\end{equation}
	and also to establish appropriate conditions under which
	\begin{equation}\label{1-cluster}
		\{S_n^{\dagger}A_n\}_n \sim_{\sigma}~{1}, \quad \text{or} \quad \{A_nS_n^{\dagger}\}_n \sim_{\sigma}~{1}.
	\end{equation}
	
	\begin{theorem}\label{cluster_difference}
		{Given \( A_n \) as in \eqref{eq:A_Toeplitz}, suppose that the sizes \( n_j \), \( j = 1, \ldots, \nu \), satisfy Assumptions \ref{assump1} and \ref{assump2}, in which we assume $c_j=\frac{\alpha_j}{\beta_j}$ with $\alpha_j,\beta_j\in\mathbb{N}^+$ for $j=1,\ldots,\nu$. Then, it holds
        \begin{equation*}
            \{S_n - A_n\}_n \sim_{\sigma}~{0},
        \end{equation*}
        where \( {S}_n \) is defined in (\ref{eq:S_n}). If moreover $S_n$ and $A_n$ are Hermitian for all $n$, then
        \begin{equation*}
            \{S_n - A_n\}_n \sim_{\lambda}~{0}.
        \end{equation*}}
	\end{theorem}
	\begin{proof}
		Employing the extra-dimensional approach outlined in Theorem~\ref{Extradimensional}, we can assume, without loss of generality, that the size of each block is $n_i=c_i n$. {Using the notation from Theorem \ref{block_symbols}, where $c_i=\frac{m_i}{m}$ with $m={\rm lcm}(\beta_1,\ldots,\beta_\nu)$, the matrices $\widehat A_n$ in \eqref{eq:hat A_Toeplitz} and $S_n$ take} the following straightforward forms
		\begin{small}
			\begin{equation}\label{eq:A_Toeplitz-equal}
				\widehat A_n=\left[
				\begin{tikzpicture}[baseline=(m.center)]
					\matrix (m) [matrix of math nodes,column sep=0.15em,row sep=0.15em] {
						|[draw,dashed, minimum width=2cm, minimum height=2cm]| T_{n(m)}(E_{1,1}) & |[draw,dashed, minimum width=2.4cm, minimum height=2cm]| T_{n(m)}(E_{1,2}) & \cdots & |[draw,dashed, minimum width=1.6cm, minimum height=2cm]| T_{n(m)}(E_{1,{\nu}}) \\
						|[draw,dashed, minimum width=2cm, minimum height=2.4cm]| T_{n(m)}(E_{2,1}) & |[draw,dashed, minimum width=2.4cm, minimum height=2.4cm]| T_{n(m)}(E_{2,2}) & \cdots & |[draw,dashed, minimum width=1.6cm, minimum height=2.4cm]| T_{n(m)}(E_{2,{\nu}}) \\
						\vdots & \vdots & \ddots & \vdots \\
						|[draw,dashed, minimum width=2cm, minimum height=1.6cm]| T_{n(m)}(E_{{\nu},1}) & |[draw,dashed,minimum width=2.4cm, minimum height=1.6cm]| T_{n(m)}(E_{{\nu},2}) & \cdots & |[draw,dashed, minimum width=1.6cm, minimum height=1.6cm]| T_{n(m)}(E_{\nu,\nu}) \\
					};
				\end{tikzpicture}
				\right]
			\end{equation}
		\end{small}
		\noindent and
		\begin{small}
			\begin{equation}\label{eq:S_Toeplitz-equal}
				S_n=\left[\begin{tikzpicture}[baseline=(m.center)]
					\matrix (m) [matrix of math nodes,column sep=0.15em,row sep=0.15em] {
						|[draw,dashed, minimum width=2cm, minimum height=2cm]| S_{n(m)}(E_{1,1}) & |[draw,dashed, minimum width=2.4cm, minimum height=2cm]| S_{n(m)}(E_{1,2}) & \cdots & |[draw,dashed, minimum width=1.6cm, minimum height=2cm]| S_{n(m)}(E_{1,{\nu}}) \\
						|[draw,dashed, minimum width=2cm, minimum height=2.4cm]| S_{n(m)}(E_{2,1}) & |[draw,dashed, minimum width=2.4cm, minimum height=2.4cm]| S_{n(m)}(E_{2,2}) & \cdots & |[draw,dashed, minimum width=1.6cm, minimum height=2.4cm]| S_{n(m)}(E_{2,{\nu}}) \\
						\vdots & \vdots & \ddots & \vdots \\
						|[draw,dashed, minimum width=2cm, minimum height=1.6cm]| S_{n(m)}(E_{{\nu},1}) & |[draw,dashed,minimum width=2.4cm, minimum height=1.6cm]| S_{n(m)}(E_{{\nu},2}) & \cdots & |[draw,dashed, minimum width=1.6cm, minimum height=1.6cm]| S_{n(m)}(E_{\nu,\nu}) \\
					};
				\end{tikzpicture}
				\right],
			\end{equation}
		\end{small}
		where $n(m)=\frac{n}{m}$ and where the functions $E_{j,k}$, $j,k=1,\ldots,\nu$, are those defined explicitly in Theorem \ref{th:general}. Next, we break down the difference $S_n - A_n$ into two parts
		\[
		S_n - A_n = (S_n - \widehat{A}_n) + (\widehat{A}_n - A_n).
		\]
		If both $\{S_n - \widehat{A}_n\}_n$ and $\{\widehat{A}_n - A_n\}_n$ are zero-distributed in terms of singular values, then we conclude that $\{S_n - A_n\}_n \sim_\sigma 0$. By recalling (\ref{1st zero-distr}) and (\ref{2nd zero-distr}), we already know that
		\[
		\{\widehat{A}_n - A_n\}_n \sim_\sigma 0.
		\]
		Thus we need to prove that $\{S_n - \widehat{A}_n\}_n \sim_\sigma 0$.
		By construction, we have $\{L_{n(m)}(E_{i,j})\}_n := \{S_{n(m)}(E_{i,j}) - T_{n(m)}(E_{i,j})\}_n \sim_{\sigma}~{0}$. Using Theorem~\ref{thm:zero-characterization}, for each sequence $\{L_{n(m)}(E_{i,j})\}_n$ there exist matrix-sequences  $\{R_{n(m)}^{(i,j)}\}_n$ and $\{N_{n(m)}^{(i,j)}\}_n$ such that
		\[
		L_{n(m)}(E_{i,j}) = R_{n(m)}^{(i,j)} + N_{n(m)}^{(i,j)},
		\]
		where
		\[
		{\operatorname{rank} \big(R_{n(m)}^{(i,j)}\big)} = o(n(m)) \quad \text{and} \quad \lim_{n \to \infty} \|N_{n(m)}^{(i,j)}\| = 0.
		\]
		Next, we define
		\[
		R_n = \begin{bmatrix}
			R_{n(m)}^{(1,1)} & \dots & R_{n(m)}^{(1,\nu)} \\
			\vdots & \ddots & \vdots \\
			R_{n(m)}^{(\nu,1)} & \dots & R_{n(m)}^{(\nu,\nu)}
		\end{bmatrix}, \quad
		{N_{n}} = \begin{bmatrix}
			N_{n(m)}^{(1,1)} & \dots & N_{n(m)}^{(1,\nu)} \\
			\vdots & \ddots & \vdots \\
			N_{n(m)}^{(\nu,1)} & \dots & N_{n(m)}^{(\nu,\nu)}
		\end{bmatrix}.
		\]
		Hence
		\[
		S_n - \widehat{A}_n = R_n + N_n.
		\]
		We prove the theorem using {the elementary block decompositions} $R_n=\sum_{i,j=1}^{\nu}El_{i,j} \otimes R_{n(m)}^{(i,j)}$ and $N_n=\sum_{i,j=1}^{\nu}El_{i,j} \otimes N_{n(m)}^{(i,j)}$, with $El_{i,j} \in \mathbb{C}^{\nu \times \nu}$ as the $(i,j)$-elemental matrix, with all zero entries except one in position $(i,j)$.
		Indeed, note that the number of blocks $\nu^2$ is fixed by hypothesis and by sublinearity. {Hence,}
		\[
		\operatorname{rank}(R_n) \leq \sum_{i,j=1}^\nu \operatorname{rank}(R_{n(m)}^{(i,j)}),
		\]
		and
		\[
		\lim_{n(m) \to \infty} \frac{\operatorname{rank}(R_{n(m)}^{(i,j)})}{n(m)} = 0.
		\]
		Therefore
		\[
		\lim_{n \to \infty} \frac{\operatorname{rank}(R_n)}{n} = 0.
		\]
		In a similar way, we have
		\[
		\lim_{n \to \infty} \|N_n\| \leq \lim_{n(m) \to \infty} \sum_{i,j=1}^\nu \|N_n(m)^{(i,j)}\|=0,
		\]
		which implies
		\[
		\lim_{n \to \infty} \|N_n\| = 0.
		\]
		By combining the properties of $R_n$ and $N_n$, we conclude
		\[
		\{S_n - \widehat{A}_n\}_n \sim_\sigma~{0}.
		\]
		Given that $\{\widehat{A}_n - A_n\}_n$ and $\{S_n - \widehat{A}_n\}_n$ are zero-distributed in terms of singular values, it follows that
		\[
		\{S_n - A_n\}_n = \{(S_n - \widehat{A}_n) + (\widehat{A}_n - A_n)\}_n \sim_\sigma~{0}.
		\]
		If $\{S_n\}_n$ and $\{A_n\}_n$ are also Hermitian, then the claim on the eigenvalues clustering follows by axioms \textbf{GLT \ref{glt1-distributions}} and \textbf{GLT \ref{glt3-algebra}}, since $\{S_n - A_n\}_n$ is a zero-distributed Hermitian matrix-sequence.
	\end{proof}
	We also have a strong clustering result based on the continuity of the spectral symbol of $A_n$ and the strong clustering of the blocks of $\widehat{A}_n$ and $S_n$.
	
	\begin{theorem}\label{proper_cluster}
		Suppose $A_n$, $\widehat{A}_n$, and $S_n$ have the block structure {described} in detail in Theorem \ref{cluster_difference}. Assume further that $\{\widehat{T}_{n^{'}}-S_{n^{'}}\}_{n^{'}}$ and $\{X_{n^{'}}-R_{n^{'}}\}_{n^{'}}$ have a strong cluster {at} $0$ in the singular values sense, where $X_{n'}$ and $R_{n'}$ are the remainders described in (\ref{eq:hat A_Toeplitz}), (\ref{approximation_S1}) and (\ref{approximation_S2}),  and
		\begin{equation}
			\{A_n\}_n \sim_{\sigma}~{F},
		\end{equation}
		with $F$ continuous. Then $\{S_n - A_n\}_n$ has a strong {singular value cluster at} 0.
	\end{theorem}
	\begin{proof}
		The proof follows similar lines to the proof of Theorem~\ref{cluster_difference}. We aim to prove that $\{\widehat{A}_n -A_n\}_{{n}}$ and $\{S_n - \widehat{A}_n\}_{{n}}$ have a singular value cluster at $0$. First we show that \( \{K_n\}_{{n}} \), $ K_n= A_n - \widehat{A}_n$, is strongly clustered at $0$ in terms of singular values. This can be seen by inspecting the approximation process to obtain $\widehat{A}_n$ from $A_n$. Indeed, recall that \( \widehat{A}_n \) is constructed from \( A_n \) through two approximation stages:
		\begin{enumerate}
			\item The approximation of off diagonal blocks by adjusting off diagonal Toeplitz block sizes resulting in \( \widetilde{A}_n \);
			\item The approximation process of square Toeplitz matrices to form \( \widehat{A}_n \) from \( \widetilde{A}_n\).
		\end{enumerate}
		Therefore, if $K_n^{(1)}$ is the discrepancy given by the first step and $K_n^{(2)}$ the discrepancy given by the second step, then
		\begin{displaymath}
			K_n = A_n - \widehat{A}_n = (A_n - \widetilde{A}_n) + (\widetilde{A}_n - \widehat{A}_n) = K_n^{(1)} + K_n^{(2)}.
		\end{displaymath}
		The approximation processes are described in detail in (\ref{eq:tilde A_Toeplitz}) and (\ref{eq:hat A_Toeplitz}). Consider the first stage: the difference between the original and adjusted blocks is
		\begin{displaymath}
			K_{n_i,n_j}^{(1)} = T_{n_i,n_j}(f_{i,j}) - \widetilde{T}_{n_i,n_j}(f_{i,j}).
		\end{displaymath}
		
		The discrepancy \( K_{n_i,n_j}^{(1)} \) is non-zero exclusively in the sections of the blocks where zeros have been inserted. It was observed in the proof of \cite[Lemma~3.2]{prequel} that these sections correspond either to the expressions $(Y_{n_i} \otimes I_s) X_{n_i,n_j}(f_{i,j})$ or $X_{n_i,n_j}(f_{i,j}) (Y_{n_j} \otimes I_t)$. Here, \( Y_{n_i} \) and \( Y_{n_j} \) are block exchange matrices, while \( X_{n_i,n_j}(f_{i,j}) \) represents a submatrix of the rectangular Hankel matrix \( H_{(n_i-n_j),n_j}(f_{i,j}) \) or \( H_{n_i,(n_j-n_i)}(f_{i,j}) \).
		Applying Remark \ref{hankel_continuous} 
        and the continuity of $f_{i,j}$, we deduce that the singular values of matrix-sequences of form $\{K_{n_i,n_j}^{(1)}\}_{{n}}$ strongly cluster at $0$. Therefore, \( \{K_n\}_{{n}} \) has also a strong cluster at $0$ using the elementary block decomposition as used in Theorem \ref{cluster_difference}.

		During the second phase, the blocks \( T_{n_j}(f) \) within \( \widetilde{A}_n \) are approximated by employing \( {m_j} \) copies of smaller Toeplitz matrices \( {T_{n_m}(f)} \) related to the same symbol. This estimation relies on the hypothesis \( {n_j = m_j n_{m} + o(n_m)} \). Further explanation of this approximation is provided in (\ref{eq:hat A_Toeplitz}).

        {Since} $f$ is a fixed trigonometric polynomial of degree $r$, {the} discrepancies are limited to small corner sections between Toeplitz blocks, resulting in a low-rank matrix dependent on $r$ rather than $n$. Assume {that} $f$ is continuous and, for a given $\epsilon > 0$, consider a trigonometric polynomial $p_\epsilon$ with $\| f - p_{\epsilon} \|_{L^\infty} < \epsilon$. Given that $\widehat{T}_{n_j}(f)$ is a block diagonal matrix with Toeplitz blocks, the linearity of Toeplitz operators implies
		\begin{align}\label{K2-approx}
			K_{n_j}^{(2)} &= T_{n_j}(f) - \widehat{T}_{n_j}(f)=\\
			&=\left[T_{n_j}(p_{\epsilon}) - \widehat{T}_{n_j}(p_{\epsilon})  \right] + \left[T_{n_j}(f -p_{\epsilon}) - \widehat{T}_{n_j}(f - p_{\epsilon})  \right].
		\end{align}
		
		Applying the preceding remark and Theorem~\ref{thm:bound_toeplitz}, (\ref{K2-approx}) takes the form
		\begin{equation}
			K_{n_j}^{(2)}=R_{n_j}^{(\epsilon)}+N_{n_j}^{(\epsilon)},   
		\end{equation}
where the rank of $R_{n_j}^{(\epsilon)}$ depends solely on $\epsilon$, with $\lim_{\epsilon \to 0} \|N_{n_j}^{(\epsilon)}\| = 0$
and $\lim_{n_j \to \infty} \|N_{n_j}^{(\epsilon)}\| = \| f - p_{\epsilon} \|_{L^\infty}$. This suffices to see that $\{K_{n_j}^{(2)}\}_{{n}}$ has a strong singular value cluster at $0$ and consequently $\{K_n^{(2)}\}_{{n}}$ as well. Therefore, the discrepancy \( \{K_n\}_{{n}}\}_{{n}} \) is representable as a sum of matrix-sequences strongly clustered at $0$ in singular value.\\
		The proof that $\{S_n - \widehat{A}_n \}_n$ is strongly clustered at $0$ parallels that of Theorem~\ref{cluster_difference}. If $\widehat{A}_n$ and $S_n$ have block remainders, note that $\{S_n - \widehat{A}_n\}_{{n}}$ contains strongly zero-distributed blocks by hypothesis. Thus, we can use the arguments of elementary sum decomposition from Theorem~\ref{cluster_difference}.
	\end{proof}
	
	\begin{theorem}\label{cluster_inverse}
		Assume the same hypothesis and definitions of $A_n$ and $S_n$ as in Theorem~\ref{cluster_difference}. Furthermore, assume that $\{A_n\}_n \sim_{\sigma}~{F}$, where the structure of $F$ is described in Theorem~\ref{th:general},  and all singular values of $F$ are non-zero almost everywhere. The following implications hold:
		\begin{enumerate}
			\item If $s \ge t,$ then
			\begin{equation}
				\{S_n^{\dag}A_n\}_n \sim_{\sigma}~{1}.
			\end{equation}
			
			\item If $s \le t,$ then
			\begin{equation}
				\{A_nS_n^{\dag}\}_n \sim_{\sigma}~{1}.
			\end{equation}
		\end{enumerate}
		Furthermore, if $\{A_n\}_n$ is Hermitian matrix-sequence and $\{S_n \}_n$ is Hermitian and positive definite matrix-sequence, then $\{S_n^{-1}A_n\}_n \sim_{\lambda}~{1}$.
	\end{theorem}
	\begin{proof}
		Suppose we have the first case $s\ge t$. The proof of the second case is equal. First, we assume again without loss of generality that $n_j=c_jn$, thus $\widehat{A}_n$ and $S_n$ have the structure detailed respectively in (\ref{eq:A_Toeplitz-equal}) and (\ref{eq:S_Toeplitz-equal}), by using the extra-dimensional argument.

		In Theorem~\ref{cluster_difference}, it was shown that
		\begin{equation}
			\{S_n-A_n\}_n \sim_{\sigma}~{0},
		\end{equation}
		and, naturally, we have $ \{\widehat{A}_n-A_n\}_n \sim_{\sigma}~{0}$ and $ \{\widehat{A}_n-S_n\}_n \sim_{\sigma}~{0}.$
		Thus, by using the characterization of zero-distributed sequences, we can write,
		\begin{align}\label{zero_perturbation}
			S_n&=\widehat{A}_n+R_n+Z_n, \\ \label{zero_perturbation 2}
			A_n&=\widehat{A}_n+\overline{R}_n+\overline{Z}_n,
		\end{align}
		where
		\begin{align}
			&\lim_{n \to \infty} \frac{R_n}{n}=\lim_{n \to \infty} \frac{\overline{R_n}}{n}=0\\
			&\lim_{n \to \infty} \|Z_n\| = \lim_{n \to \infty} \|\overline{Z}_n\| = 0.
		\end{align}
		Let $D_n= \Pi_s C_n \Pi_t^T$ and $G_n= \Pi_s A_n \Pi_t^T$. By applying Theorem \ref{equal-blocks-basic} and by exploiting (\ref{zero_perturbation})-(\ref{zero_perturbation 2}), we obtain
		\begin{align}
			&D_n=T_n(F)+\Pi_s R_n \Pi_t^T + \Pi_s Z_n \Pi_t^T,\\
			&G_n=T_n(F)+\Pi_s \overline{R}_n \Pi_t^T + \Pi_s \overline{Z}_n \Pi_t^T.
		\end{align}
		It is important to note that $\Pi_s R_n \Pi_t^T + \Pi_s Z_n \Pi_t^T$ and $\Pi_s \overline{R}_n \Pi_t^T + \Pi_s \overline{Z}_n \Pi_t^T$ are again zero-distributed. Since $\{T_n(F)\}_n \sim_{\rm GLT} F$, by axiom \textbf{GLT \ref{glt2-buildingblocks}}, using  Theorem~\ref{thm:zero-characterization} and axiom \textbf{GLT \ref{glt3-algebra}}, we also have
		\begin{equation}
			\{D_n\}_n \sim_{\mathrm{GLT}}~{F}\ \  \ \text{ and }\ \ \ \{G_n\}_n \sim_{\mathrm{GLT}}~{F}.
		\end{equation}
		By hypothesis, $F$ has non zero singular values almost everywhere, thus we can apply axiom \textbf{GLT \ref{glt3-algebra}} once again, obtaining
		\begin{equation}
			\{D_n^{\dag} G_n \}_n \sim_{\mathrm{GLT}}~{I_t}.
		\end{equation}
		By axiom \textbf{GLT \ref{glt1-distributions}} we have $$ \{D_n^{\dag} G_n \}_n \sim_{\sigma} I_t.$$
		The singular value distribution described by the identity matrix symbol signifies a (weak) clustering at $1$, thus we can write $\{D_n^{\dag} G_n \}_n \sim_{\sigma} 1.$

		Finally, $\{\Pi_tD_n^{\dag} G_n\Pi_t^T\}_n = \{S_n^{\dag} A_n\}_n \sim_{\sigma} 1$, this is what we wanted to prove.
		The claim on the eigenvalue distribution follows again by algebraic manipulation of $D_n$ and $G_n$ since the related matrix-sequences are of GLT type. Indeed if $S_n$ and $A_n$ are Hermitian and Hermitian positive definite, respectively, then the same character is inherited by $D_n$ and $G_n$, respectively. Therefore we can write
		\begin{equation}
			\{D_n^{-1} G_n \}_n =  \{D_n^{-\frac{1}{2}} G_n D_n^{-\frac{1}{2}} \}_n \sim_{\lambda}~{I_{t}}.
		\end{equation}
		By using axiom \textbf{GLT \ref{glt3-algebra}} and by invoking the fact {that} $\{D_n^{-\frac{1}{2}} G_n D_n^{-\frac{1}{2}} \}_n$, $\{D_n\}_n$ and $\{G_n\}_n$ are Hermitian sequences, by similarity arguments, we directly infer $\{S_n^{\dag} A_n\}_n \sim_{\lambda}~{1}$.
	\end{proof}
	
	\section{Numerical Tests}\label{sec:numerical_tests}
	The present section is divided into three parts. While the first two deal with various $\nu,s,t$, {and} polynomial generating functions, so that each block in (\ref{eq:A_Toeplitz}) is Toeplitz and banded, the examples in the third part concern dense matrices coming from the approximation of fractional operators. Furthermore, the last part {also contains} a two-level example coming from fractional operators in two dimensions, {thus} opening the door to generalizations of the findings of the previous sections to the multilevel setting: we emphasize that so far even the basic spectral analysis of matrix-sequences of the form (\ref{eq:A_general}) where the blocks show multilevel structures is an open problem not covered in \cite{pre-prequel,prequel,conj-blo-I}.
	
	\subsection{Matrix-sequences and example groups}\label{sec:matrix_sequences}
	In the present section, we define the matrix-sequences $\{A_n\}_n$ used in our numerical experiments. For each $n$, $A_n$ is a block structure with Toeplitz blocks as described in (\ref{eq:A_Toeplitz}). We use the same set of selected examples, for experimental reproducibility purposes. We categorize them into three distinct groups based on the block structure and the nature of the generating functions. These groups cover various configurations that satisfy the assumptions outlined at the beginning of Section \ref{sec:general_block}.
	
	\begin{enumerate}
		\item \textbf{Group 1: $2 \times 2$ block case with scalar-valued $f_{i,j}$}.\\
		We have $\nu=2$, $t=s=1$, where each Toeplitz block is generated by scalar-valued trigonometric polynomials $f_{i,j}$. The sizes of the Toeplitz blocks are selected as
		\begin{enumerate}
			\item $n_1 = 3\eta$, $n_2 = 2\eta$;
			\item $n_1 = 3\eta$, $n_2 = 2\eta + 20$;
			\item $n_1 = 3\eta$, $n_2 = 2\eta + \lceil{\sqrt{\eta}\,}\rceil$.
		\end{enumerate}
		
		\item \textbf{Group 2: $2 \times 2$ block case with rectangular matrix-valued $f_{i,j}$}.\\
		We have $\nu=2$, $t=1, s=2$, where each Toeplitz block is generated by matrix-valued trigonometric polynomials $f_{i,j}$. The sizes of the Toeplitz blocks are selected as
		\begin{enumerate}
			\item $n_1 = \eta$, $n_2 = 2\eta$;
			\item $n_1 = \eta$, $n_2 = 2\eta+2$;
			\item $n_1 = \eta$, $n_2 = 2\eta + \lceil{\sqrt{\eta}\,}\rceil$.
		\end{enumerate}
		
		\item \textbf{Group 3: $3 \times 3$ block case with square matrix-valued $f_{i,j}$}.\\
		We have $\nu=3$, $t=s=2$, where each Toeplitz block is generated by matrix-valued trigonometric polynomials $f_{i,j}$. The sizes of the Toeplitz blocks are
		\begin{displaymath}
			n_1 = \eta, \quad n_2 = \frac{\eta}{2}, \quad n_3 = 2\eta - 2.
		\end{displaymath}
	\end{enumerate}
	
	\paragraph{Group 1: $2 \times 2$ block case with scalar-valued $f_{i,j}$.}
	For the setting $\nu=2$, $t=1$, $s=1$, we consider the following trigonometric polynomials:
	\begin{equation}\label{gruppo1_t1s1}
		\begin{split}
			&f_{1,1}(t) =  2 - 2\cos(t); \quad f_{2,2}(t) =  2 - 2\cos(t) - 6\cos(2t); \\
			&f_{1,2}(t) =  1 - 2\cos(t); \quad f_{2,1}(t) = 1 - 2\cos(t).
		\end{split}
	\end{equation}
According to the notation in (\ref{eq:A_Toeplitz}), the block matrix $A_n$ is constructed as
	\begin{displaymath}
		A_n = \begin{bmatrix}
			T_{n_1}(f_{1,1}) & T_{n_1,n_2}(f_{1,2}) \\
			T_{n_2,n_1}(f_{2,1}) & T_{n_2}(f_{2,2})
		\end{bmatrix}.
	\end{displaymath}

	\paragraph{Group 2: $2 \times 2$ block case with rectangular matrix-valued $f_{i,j}$.}
	For the setting $\nu=2$, $s=1$, $t=2$, we consider the following rectangular matrix-valued trigonometric polynomials
	
	\begin{equation}\label{gruppo2_t2s1}
		\begin{split}
			&f_{1,1}(t) = \begin{pmatrix} 2 - 2\cos(t), & 4 + 6\cos(2t) \end{pmatrix};\\
			&f_{2,2}(t) = \begin{pmatrix} 3 + 2\cos(t), & 4 + 6\cos(t) - 2\cos(2t) \end{pmatrix}; \\
			&f_{1,2}(t) = \begin{pmatrix} 1 + {\rm e}^{\iota t}, & 1 - {\rm e}^{-\iota t} \end{pmatrix}; \quad f_{2,1}(t) = f_{1,2}(t).
		\end{split}
	\end{equation}
According to the notation in (\ref{eq:A_Toeplitz}), the block matrix $A_n$ is constructed as
	\begin{displaymath}
		A_n =
		\begin{bmatrix}
			T_{n_1}(f_{1,1}) & T_{n_1,n_2}(f_{1,2}) \\
			T_{n_2,n_1}(f_{2,1}) & T_{n_2}(f_{2,2})
		\end{bmatrix}.
	\end{displaymath}

	\paragraph{Group 3: $3 \times 3$ block case with square matrix-valued $f_{i,j}$}
	For the setting $\nu=3$, $t=2$, $s=2$, we consider matrix-valued functions $f_{i,j}$ and the associated Toeplitz matrices $T_{n_i,n_j}(f_{i,j})$. These functions are related to the discretization of differential operators. More in detail, we consider the 1D constant coefficient second-order differential equation
	
	\begin{small}
		\begin{equation}\label{FEM_problem}
			\begin{cases}
				u''(x) = \psi(x) & \text{on } (0,1), \\
				u(0) = u(1) = 0,
			\end{cases}
		\end{equation}
	\end{small}
	where $\psi(x) \in L^2\left(0,1\right)$.
	
	We select structures and functions from the following contexts:
	\begin{enumerate}
		\item \textbf{Finite Differences (FD) Discretization:} The FD discretization of (\ref{FEM_problem}) leads to the Toeplitz matrix generated by $f(\theta) = 2 - 2\cos\theta$. However, as discussed in \cite[Remark 1.3]{MR3543002}, the symbol is not unique and by viewing the matrix in $2\times 2$ blocks, the corresponding $2\times 2$ matrix-valued symbol is given by \cite{huckle}:
		\begin{displaymath}
			f^{[2]}(\theta) = \hat{f}_0^{[2]} + \hat{f}_{-1}^{[2]} {\rm e}^{-\iota\theta} + \hat{f}_1^{[2]} {\rm e}^{\iota\theta},
		\end{displaymath}
		where
		\begin{displaymath}
			\hat{f}_0^{[2]} = T_2(2 - 2\cos\theta), \quad \hat{f}_{-1}^{[2]} = -e_d e_1^T, \quad \hat{f}_1^{[2]} = (\hat{f}_{-1}^{[2]})^T = -e_1 e_d^T.
		\end{displaymath}
		
		\item \textbf{$\mathbb{Q}_2$ Lagrangian Finite Element Method (FEM) Approximation:} The scaled stiffness matrix is a Toeplitz matrix generated by ${f}_{\mathbb{Q}_2}(\theta)$ \cite{qp}:
		\begin{equation*}
			\begin{split}
				f_{\mathbb{Q}_2}(\theta) = \frac{1}{3} \left(
				\begin{bmatrix}
					16 & -8 \\
					-8 & 14
				\end{bmatrix} +
				\begin{bmatrix}
					0 & -8 \\
					0 & 1
				\end{bmatrix} {\rm e}^{\iota\theta} +
				\begin{bmatrix}
					0 & 0 \\
					-8 & 1
				\end{bmatrix} {\rm e}^{-\iota\theta}
				\right).
			\end{split}
		\end{equation*}

		\item \textbf{B-Spline Discretization:} We consider B-Spline discretizations for different degrees $p$ and regularities $k$. Specifically, we examine the pairs $(p,k) = (2,0)$ and $(3,1)$. The resulting structures $A_{p,k}$ are low-rank corrections of Toeplitz matrices generated by the following functions \cite{tom}:
		\begin{align*}
			&f^{(2,0)}(\theta) = \frac{1}{3} \left(
			\begin{bmatrix}
				4 & -2 \\
				-2 & 8
			\end{bmatrix} +
			\begin{bmatrix}
				0 & -2 \\
				0 & -2
			\end{bmatrix} {\rm e}^{\iota\theta} +
			\begin{bmatrix}
				0 & 0 \\
				-2 & -2
			\end{bmatrix} {\rm e}^{-\iota\theta}
			\right), \\
			&f^{(3,1)}(\theta) = \frac{1}{40} \left(
			\begin{bmatrix}
				48 & 0 \\
				0 & 48
			\end{bmatrix} +
			\begin{bmatrix}
				-15 & -15 \\
				-3 & -1
			\end{bmatrix} {\rm e}^{\iota\theta} +
			\begin{bmatrix}
				-15 & -3 \\
				-15 & -15
			\end{bmatrix} {\rm e}^{-\iota\theta}
			\right).
		\end{align*}
		
		\item \textbf{Multigrid Grid Transfer Operator:} We utilize the classical multigrid grid transfer operator derived from the geometric approach for finite differences and finite elements \cite{multi,MR4389580,braess,MR4284081,FRTBS}. More specifically, for degree 2, the optimal algebraic Two Grid Method (TGM) employs $P_{n,k}^{2}$ associated with
		\begin{displaymath}
			p_{\mathbb{Q}_2}(\theta) =
			\begin{bmatrix}
				\frac{3}{4} & \frac{3}{8} \\
				0 & 1
			\end{bmatrix} +
			\begin{bmatrix}
				0 & \frac{3}{8} \\
				0 & 0
			\end{bmatrix} {\rm e}^{\iota\theta} +
			\begin{bmatrix}
				\frac{3}{4} & -\frac{1}{8} \\
				1 & 0
			\end{bmatrix} {\rm e}^{-\iota\theta} +
			\begin{bmatrix}
				0 & -\frac{1}{8} \\
				0 & 0
			\end{bmatrix} {\rm e}^{-2\iota\theta}.
		\end{displaymath}
	\end{enumerate}
	
Combining the previous four functions, we construct the matrix-sequence $\{A_n\}_n$ for Group 3 as follows
	\begin{displaymath}
		A_n = \begin{bmatrix}
			T_{n_1}(f_{1,1}) & T_{n_1,n_2}(f_{1,2}) & T_{n_1,n_3}(f_{1,3}) \\
			T_{n_2,n_1}(f_{2,1}) & T_{n_2}(f_{2,2}) & T_{n_2,n_3}(f_{2,3}) \\
			T_{n_3,n_1}(f_{3,1}) & T_{n_3,n_2}(f_{3,2}) & T_{n_3}(f_{3,3})
		\end{bmatrix},
	\end{displaymath}
	where
	\begin{equation}\label{eq:example_PDE1}
		\begin{split}
			& f_{1,1}(\theta) = f_{\mathbb{Q}_2}(\theta), \quad f_{1,2}(\theta) = f_{2,1}(\theta) = f^{(2,0)}(\theta), \quad f_{1,3}(\theta) = f_{3,1}(\theta) = f^{(3,1)}(\theta); \\
			& f_{2,2}(\theta) = p_{\mathbb{Q}_2}^*(\theta) p_{\mathbb{Q}_2}(\theta) + p_{\mathbb{Q}_2}^*(\theta + \pi) p_{\mathbb{Q}_2}(\theta + \pi); \\
			& f_{2,3}(\theta) = f_{3,2}(\theta) = p_{\mathbb{Q}_2}(\theta); \quad f_{3,3}(\theta) = f^{[2]}(\theta).
		\end{split}
	\end{equation}
	
	\subsection{Numerical results on distribution, clustering, and efficiency of preconditioning}\label{sec:numerical_results}
	
	In this section, we provide numerical evidence supporting the clustering results discussed in Section \ref{sec:appr-prec}. We consider preconditioners belonging to the circulant class and of Strang type since the considered Toeplitz structures are banded, but also the Frobenius optimal choice is considered when the Strang type preconditioner is singular (see \cite{CN} and references therein). In order to visualize and verify our findings, we perform the following experiments.
	
	\begin{enumerate}
		\item \textbf{Zero clustering analysis}: we plot the singular values (or eigenvalues) of the difference between the matrix $A_n$ and $S_n$, and estimate the frequency of outliers. This experiment aims to show the clustering result found in Theorem \ref{cluster_difference}. Since we employ mostly continuous generating functions, we expect our results to show a strong clustering at $0$.
		
		\item \textbf{Clustering at $1$ of the preconditioned matrix-sequence}: we perform a similar analysis on the preconditioned matrix to assess the clustering at $1$.
		
		\item \textbf{Computational efficiency of preconditioned Krylov methods}: we evaluate the computational efficiency gains by comparing the number of iterations required for convergence in non-preconditioned versus preconditioned selected systems. In the square case, we solve linear systems of the form $A_n x = b$, and in the rectangular case, we tackle least squares problems $\min_x \|A_n x - b\|$, by using the preconditioned conjugate gradient method for the normal equations. We mostly use the GMRES method, which works in general when the matrix is not symmetric and the preconditioner is not positive definite, and we use the conjugate gradient method when the problem is symmetric and positive definite. In particular, the least square problems are solved using the normal equation method implemented with the (P)CGNE.
	\end{enumerate}

\subsubsection{Zero clustering analysis}
	
The following experiments show the zero clustering of the sequence $\{A_n - S_n\}_n $ by examining the singular values (or eigenvalues) of $A_n - S_n$. By the theory on strong clustering, we expect a converging phenomenon of most singular values (or eigenvalues) to $0$ and observe a small number of \lq\lq proper outliers'', i.e. outliers that never get close to the cluster.
	
	\paragraph{Methodology:} for each $n$, we compute the singular values of $A_n - S_n$. The singular value plot is effective for seeing the proper outliers outside the clustering region.
	
	\paragraph{Results:} as depicted in Figures from \ref{fig:zero_clustering_group1} to \ref{fig:zero_clustering_group3}, most singular values of $A_n - S_n$ are close to zero, with a negligible number of true outliers, i.e. outliers that do not seem to converge to the cluster and stay away from $0$. By inspecting the plot, these proper outliers seem to stay bounded in number and asymptotically constant. This is what we should expect from Theorem \ref{proper_cluster}.
	
	\begin{figure}
		\centering
		\includegraphics[width=.45\textwidth]{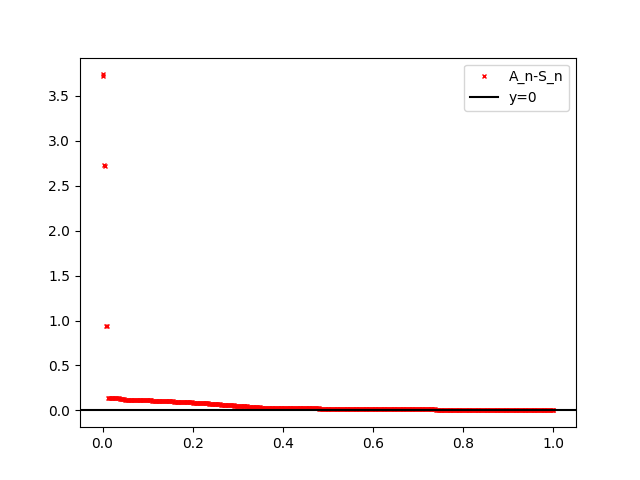}
		\includegraphics[width=.45\textwidth]{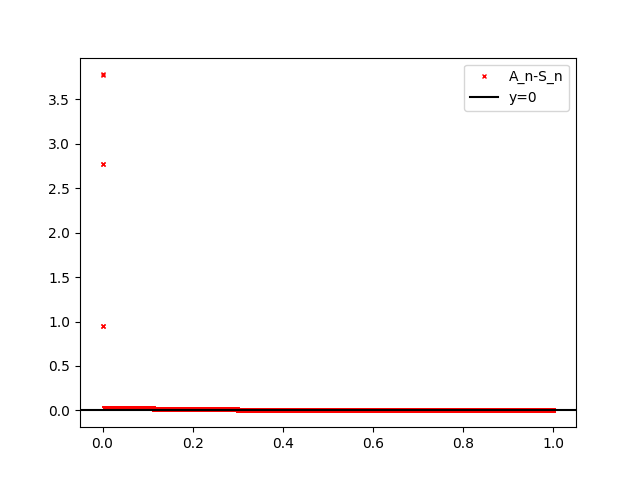} \\
             \includegraphics[width=.45\textwidth]{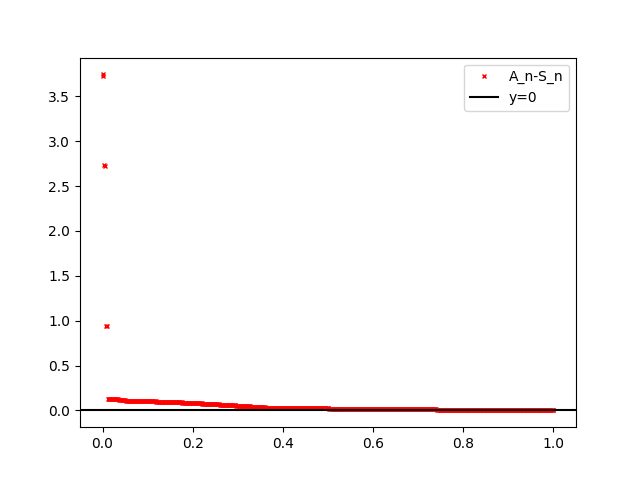}
        \includegraphics[width=.45\textwidth]{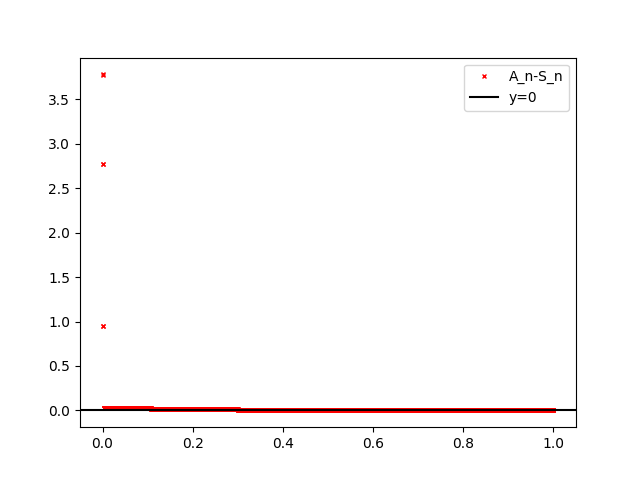} \\
       \includegraphics[width=.45\textwidth]{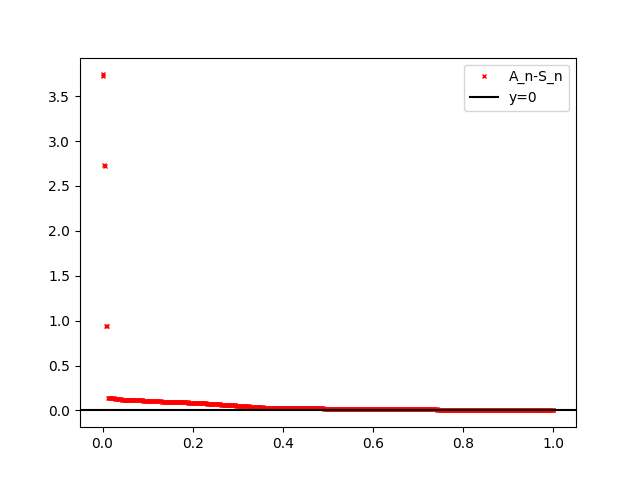}\includegraphics[width=.45\textwidth]{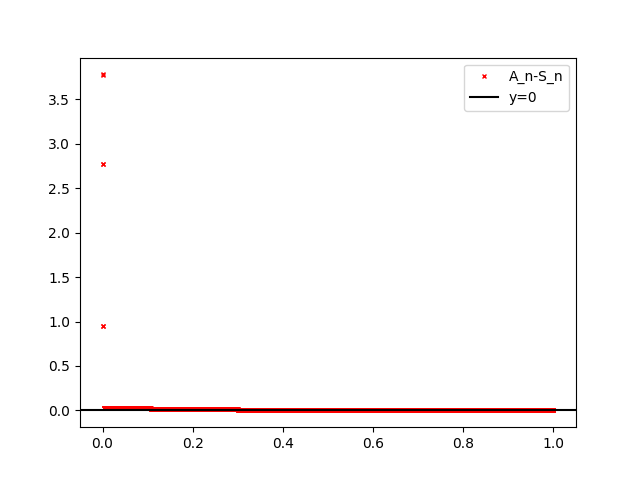}
		\caption{The cluster at $0$ of the singular values of $\{A_n-S_n\}_n$ for Group 1.
        The first row is case (a), the second row is case (b), and the third row is case (c). The first column is
        for $\eta=100$ while the second column is for $\eta=500$.}\label{fig:zero_clustering_group1}
	\end{figure}
		
	\begin{figure}
		\centering
        \includegraphics[width=.45\textwidth]{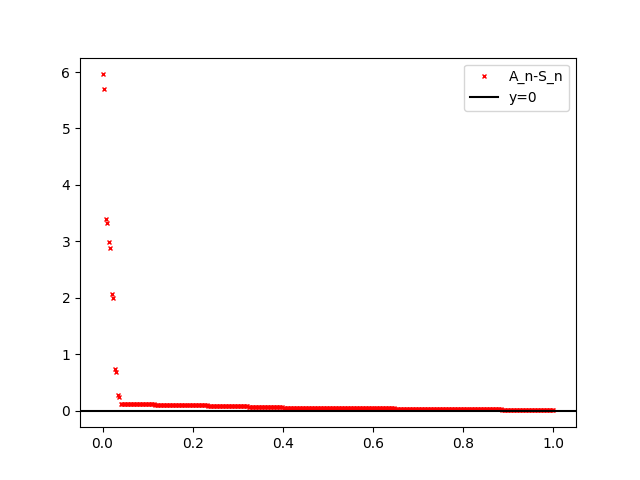}
        \includegraphics[width=.45\textwidth]{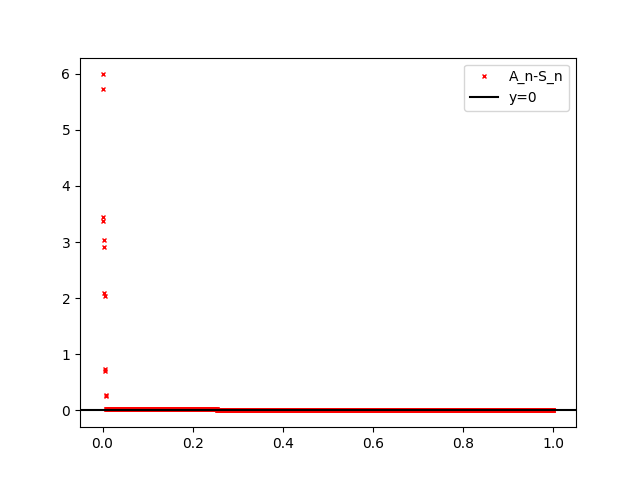} \\
		\includegraphics[width=.45\textwidth]{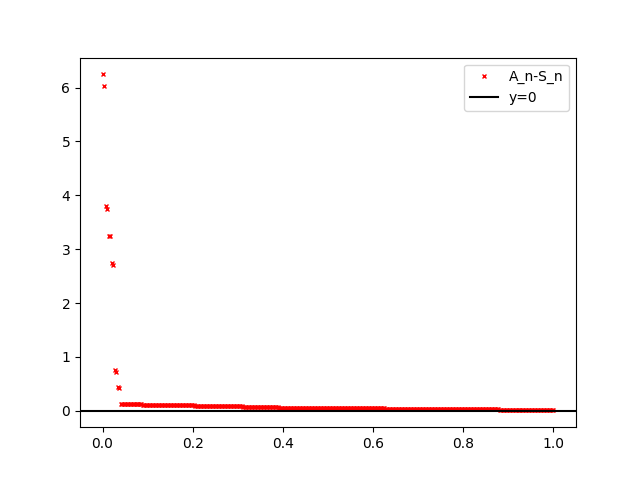}
		\includegraphics[width=.45\textwidth]{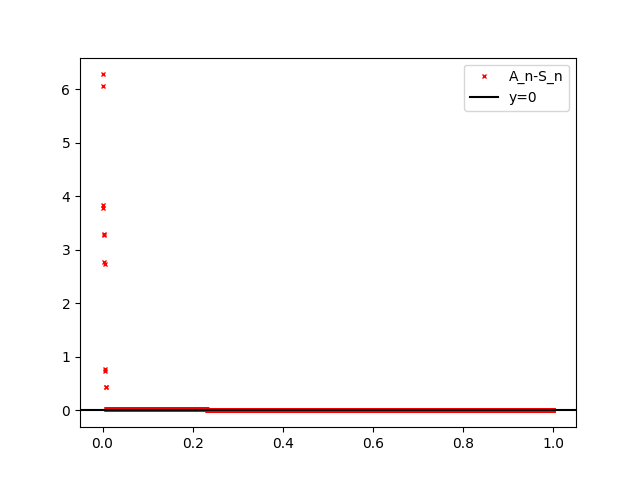} \\
        \includegraphics[width=.45\textwidth]{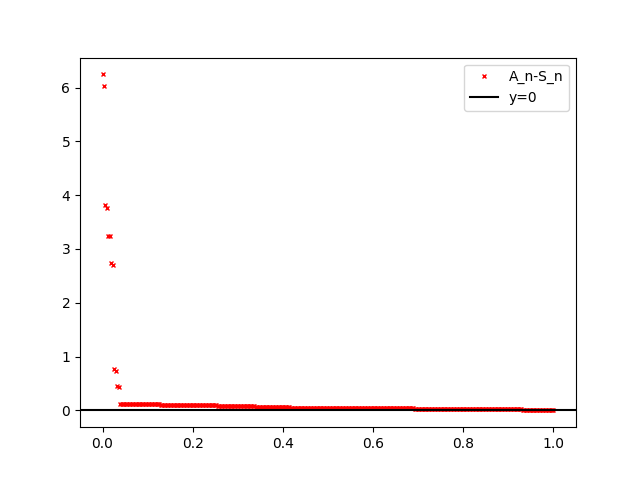}
		\includegraphics[width=.45\textwidth]{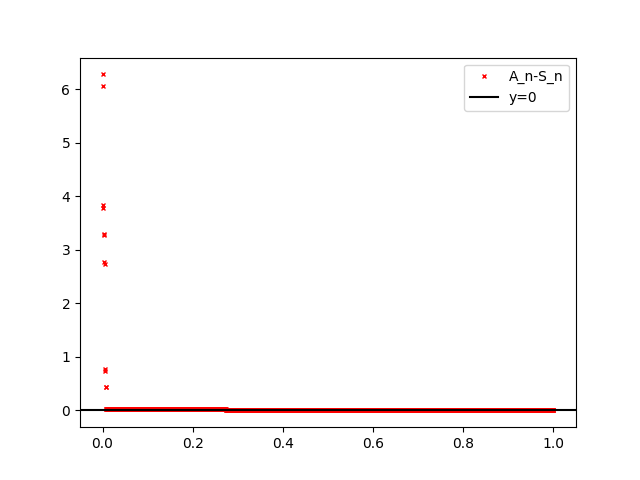}
		\caption{The cluster at $0$ of the singular values of $\{A_n-S_n\}_n$ for Group 2.
        The first row is case (a), the second row is case (b), and the third row is case (c). The first column is
        for $\eta=100$ while the second column is for $\eta=500$.}\label{fig:zero_clustering_group2}
	\end{figure}
	
	\begin{figure}
		\centering
		\includegraphics[width=.45\textwidth]{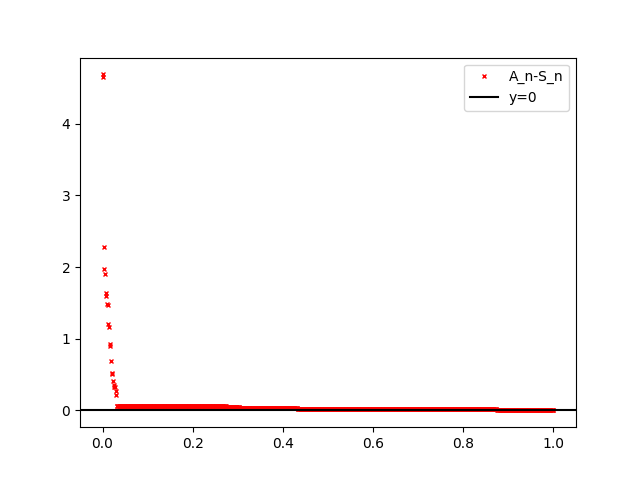}
		\includegraphics[width=.45\textwidth]{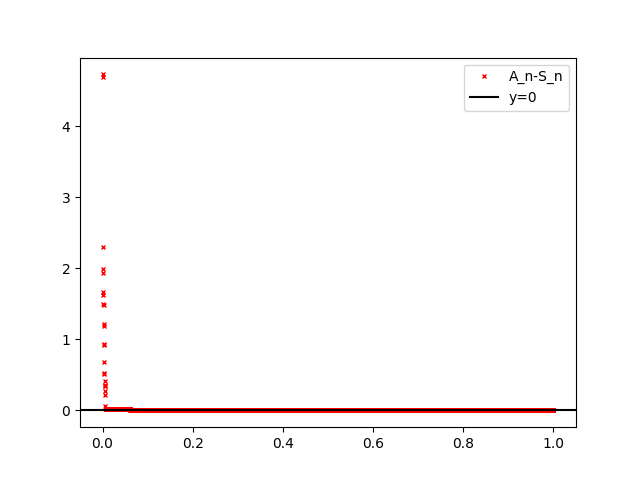}
		\caption{The cluster at $0$ of the singular values of $\{A_n-S_n\}_n$ for Group 3. The plot on the left is for $\eta=100$ while the one on the right is for $\eta=500$.}\label{fig:zero_clustering_group3}
	\end{figure}

	\subsubsection{Clustering at 1 of the preconditioned matrix-sequence}
	The second experiment
    assesses the clustering behavior at 1 of the preconditioned matrix-sequence $\{ S_n^{-1}A_n\}_n$ or
    $\{A_nS_n^{\dagger}\}_n$, by analyzing the singular values (or eigenvalues) of $S_n^{-1}A_n$ in the square case or of $A_nS_n^{\dagger}$ in the rectangular case. We compare them with {those} of the original matrix $A_n$.
	
 \paragraph{Methodology:} for each $n$, we compute singular values of the preconditioned matrix $S_n^{-1}A_n$ or $A_nS_n^{\dag}$. Similar to the previous experiment, we identify and quantify the frequency of outliers in the singular value distribution.
	
	\paragraph{Results:} Figures from \ref{fig:1_clustering_group1} to \ref{fig:1_clustering_group3} show that the singular values of the preconditioned matrix are tightly clustered around one, with very few outliers. The clustering becomes more pronounced as $n$ increases, consistent with the theoretical prediction of the clustering at $1$. This is confirmed by the results in Table \ref{tab:outliers} where the ratio between the eigenvalues (singular values) and the outliers are reported for all the considered cases. We observe that the outliers decaying to $0$ do not exceed double the number of outliers above $1$. This is reasonable by \cite{AxL,Krylov-book} since the problems we aim to solve are not exponentially ill-conditioned, and this clustering behavior allows us to solve the related linear system {quickly} by using Krylov solvers.
	
	\begin{figure}
		\centering
    \includegraphics[width=.45\textwidth]{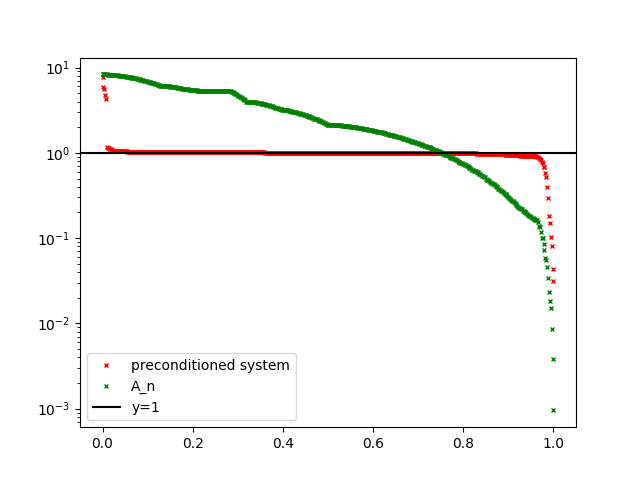}\includegraphics[width=.45\textwidth]{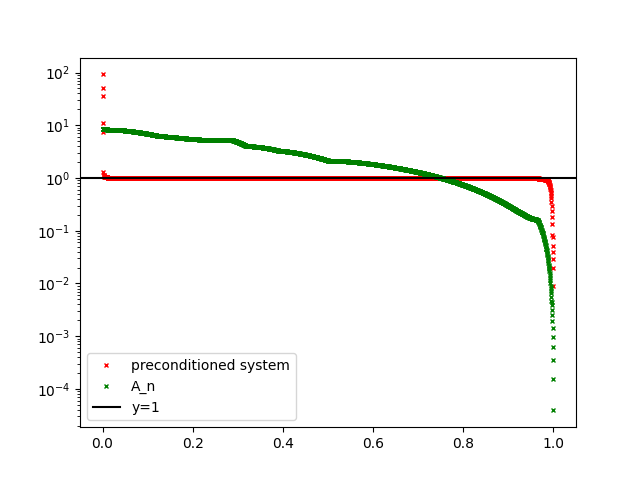} \\
    \includegraphics[width=.45\textwidth]{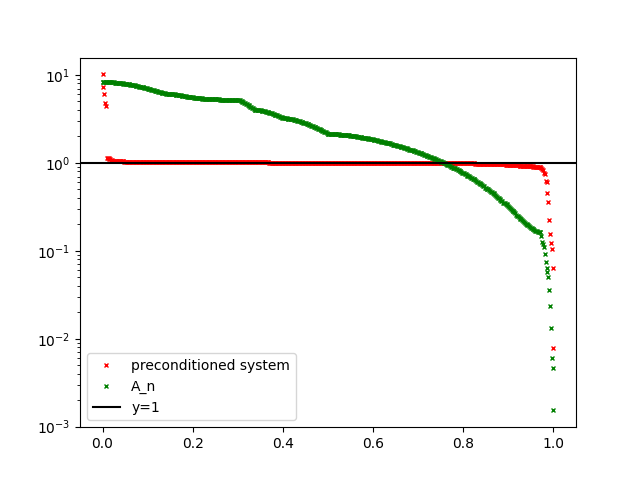}
\includegraphics[width=.45\textwidth]{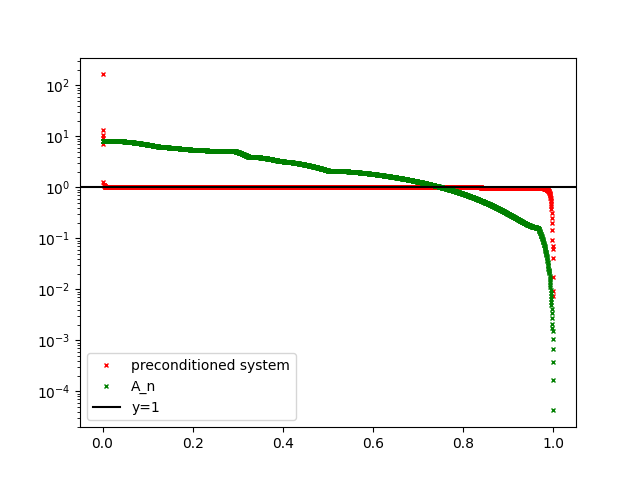} \\
\includegraphics[width=.45\textwidth]{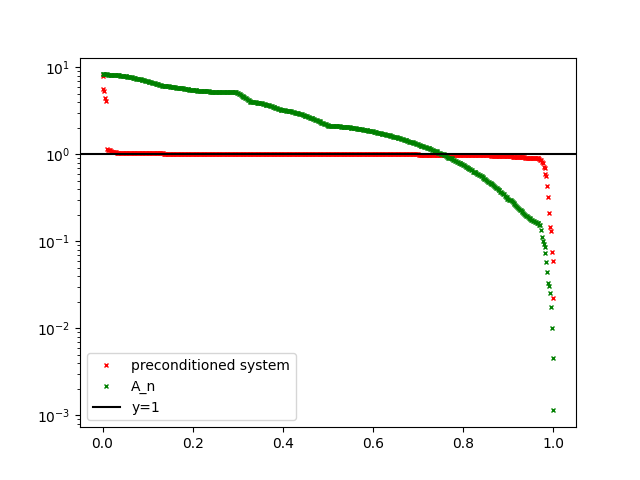}	\includegraphics[width=.45\textwidth]{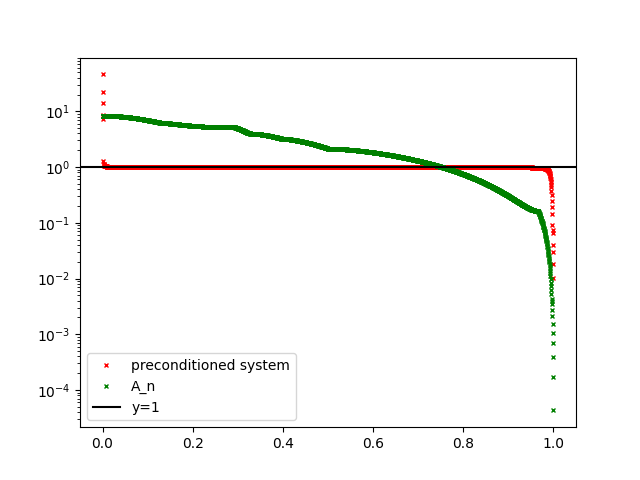}
		\caption{The cluster at $1$ of the singular values of $\{S_n^{-1}A_n\}_n$ for Group 1.
        The first row is case (a), the second row is case (b), and the third row is case (c). The first column is
        for $\eta=100$ while the second column is for $\eta=500$.}\label{fig:1_clustering_group1}
	\end{figure}
	
	\begin{figure}
		\centering		\includegraphics[width=.45\textwidth]{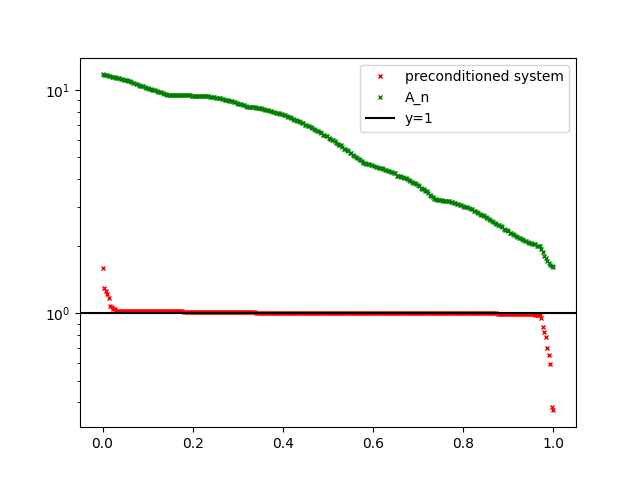}
	\includegraphics[width=.4\textwidth]{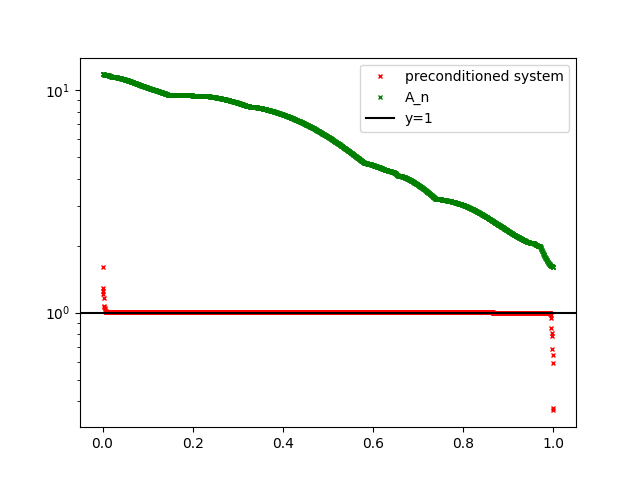}\\
    \includegraphics[width=.45\textwidth]{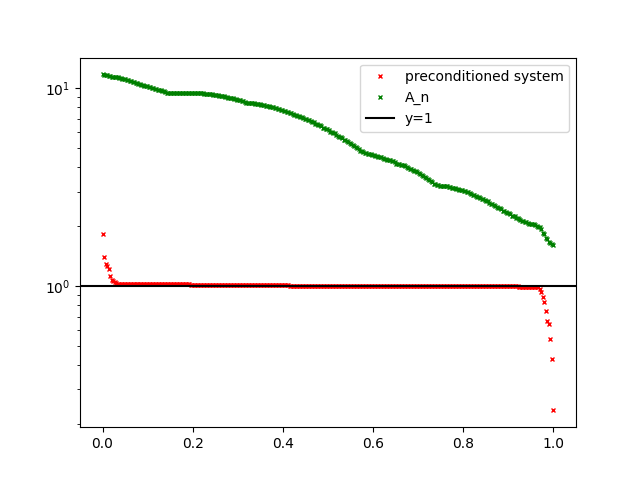}
    \includegraphics[width=.45\textwidth]{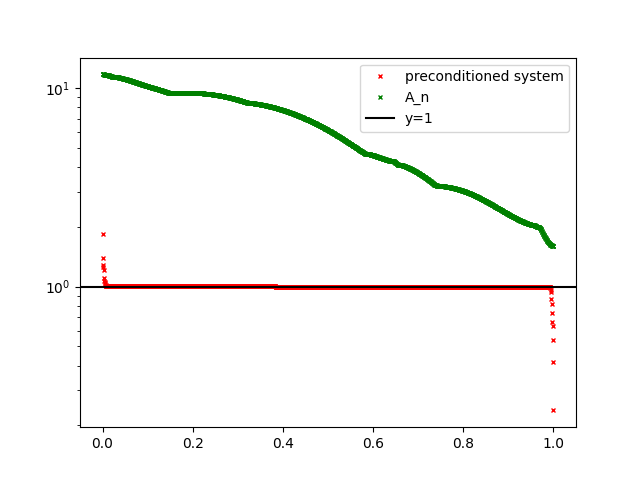}\\
    \includegraphics[width=.45\textwidth]{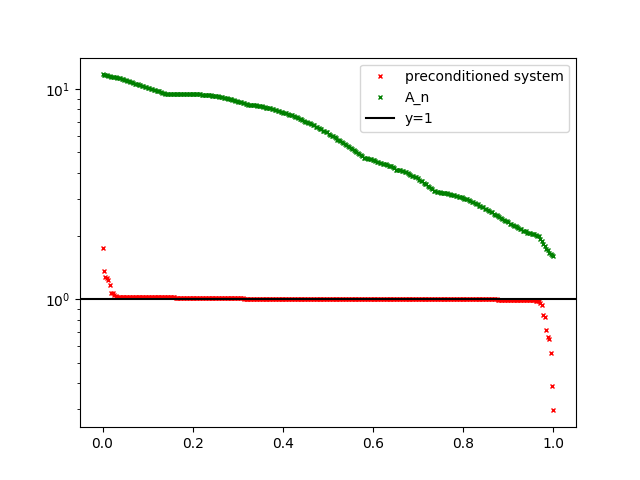}
    \includegraphics[width=.45\textwidth]{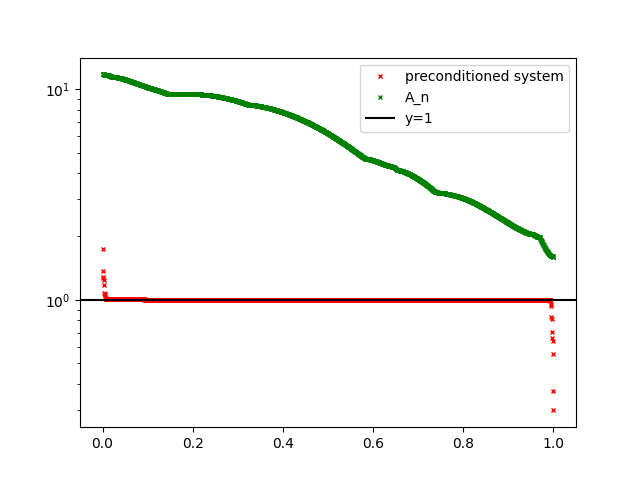}
		\caption{The cluster at $1$ of the singular values of $\{A_nS_n^{\dag}\}_n$ for Group 2.
        The first row is case (a), the second row is case (b), and the third row is case (c). The first column is
        for $\eta=100$ while the second column is for $\eta=500$.}\label{fig:1_clustering_group2}
	\end{figure}

        \begin{figure}
		\centering
        \includegraphics[width=.45\textwidth]{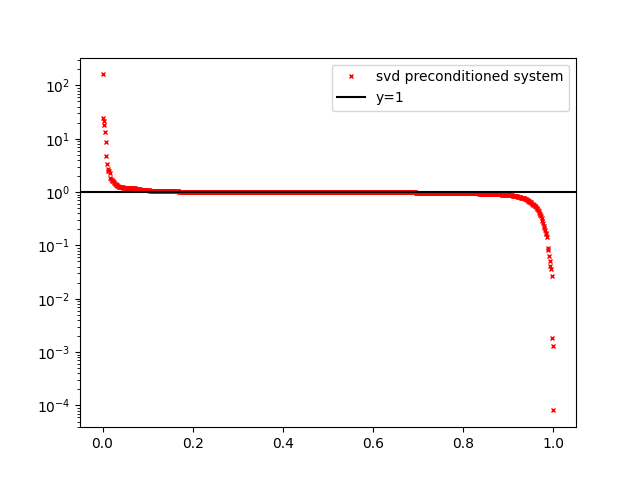}
        \includegraphics[width=.45\textwidth]{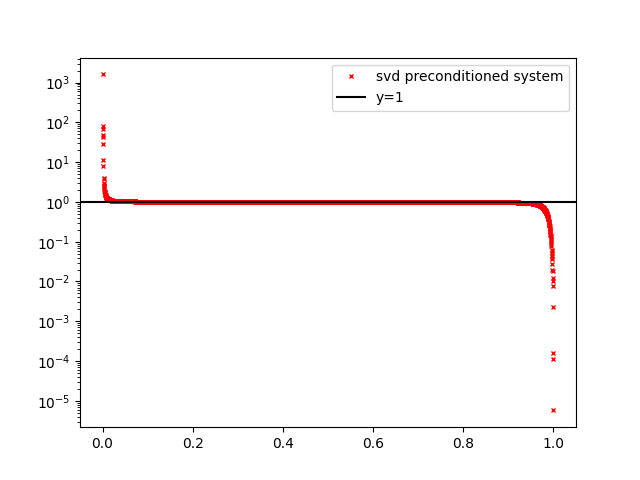}
		\caption{The cluster at $1$ of the singular values of $\{S_n^{-1}A_n\}_n$ for Group 3.
        The plot on the left is
        for $\eta=100$ while the one on the right is for $\eta=500$.}\label{fig:1_clustering_group3}
	\end{figure}
	
	\begin{table}[!h]
		\centering
		\begin{tabular}{|l|c|c|c|c|c|}
			\toprule
			Case & $\eta$ & Outliers Below & Ratio Below & Outliers Above & Ratio Above \\
			\midrule
            Group 1 &
			100 & 8 & 0.0160 & 5 & 0.010 \\
            case (a) &
			200 & 10 & 0.0100 & 5 & 0.005 \\
             &
			500 & 14 & 0.0056 & 5 & 0.002 \\
            \midrule
		 Group 1 &	100 & 8 & 0.0150 & 5 & 0.009 \\
		case (b) &	200 & 10 & 0.0090 & 5 & 0.005 \\
		&	500 & 14 & 0.0055 & 5 & 0.002 \\
        \midrule
	Group 1 &	100 & 8 & 0.0150 & 5 & 0.009 \\
        case (c)&	200 & 10 & 0.0090 & 5 & 0.005 \\
	&		500 & 14 & 0.0055 & 5 & 0.002 \\
    \bottomrule
    Group 2 & 100 & 9 & 0.030 & 8 & 0.027 \\
    case (a) & 200 & 9 & 0.015 & 8 & 0.013 \\
		&	500 & 9 & 0.006 & 7 & 0.005 \\
    \midrule
    Group 2 & 100 & 9 & 0.029 & 8 & 0.026 \\
    case (b) & 200 & 9 & 0.015 & 8 & 0.013 \\
		&	500 & 9 & 0.006 & 7 & 0.005 \\
    \midrule
    Group 2 &	100 & 9 & 0.029 & 8 & 0.026 \\
    case (c) &	200 & 9 & 0.015 & 8 & 0.013 \\
		&	500 & 9 & 0.006 & 7 & 0.005 \\
			\bottomrule
    	& 100 & 27 & 0.038 & 18 & 0.026 \\
    Group 3 & 200 & 35 & 0.025 & 17 & 0.012 \\
		& 500 & 53 & 0.015 & 23 & 0.006 \\
			\bottomrule
		\end{tabular}
		\caption{Outliers of the preconditioned matrix with respect to the cluster at $1$}
		\label{tab:outliers}
	\end{table}

\subsubsection{Computational efficiency of preconditioned Krylov methods}
The third experiment assesses the computational efficiency of preconditioned versus non-preconditioned {Krylov methods}. In particular, we test the preconditioned GMRES method and the PCG method. For rectangular matrices, we apply the (P)CGNE method for the normal equations to the underdetermined system. We compare the efficiency of the solution {to the} unpreconditioned system.
	
	\paragraph{Methodology:} We compare the behavior of the non-preconditioned system with that of the preconditioned one. The initial guess is always the zero vector. For (P)GMRES, the restarting is always set to 100 iterations.
    	In detail, the experiments we perform are as follows:
	\begin{enumerate}
		\item For group 1, we test (P)GMRES for all configurations of $\eta$.
		\item For group 2, we test (P)CGNE applied to the relative least-squares problem for case~(c).
		\item For group 3, we modify the original matrix. The matrix is symmetrized, then the diagonal block $2 \times 2$ is scaled by a factor of $1.78$ and the off-diagonal blocks by a factor of $0.7$. This allows us to test CG for a Hermitian positive definite system.
	\end{enumerate}
	In all experiments, we conventionally say that the method converges when the norm of the residual at the end of the step is less than $10^{-8}$ and the number of steps is less than $1000$. We choose the number of steps conventionally for all experiments: since we always choose $\eta=100,200,500$, a convergence of more than 1000 iterations is not in our interest because it means that the convergence remains linear and we do not have any improvement.
	
	We report the number of iterations of each experiment and the norm of the residual before the convergence goal.
	
	\paragraph{Results:} The tables \ref{tab:iteration_counts}--\ref{tab:iteration_counts_group3} show, in general, an improvement of the Krylov method when we apply the minimal {Frobenius} circulant preconditioner.
	
	In group 1 we have a symmetric indefinite matrix. The unpreconditioned method fails to converge within the desired number of iterations. Differently, the preconditioned GMRES using the preconditioner from the minimal {Frobenius} circulant cut the iteration number showing a fast convergence compared to the conditioning of the problem (see Table \ref{tab:iteration_counts}).
	
	The group 2 is well conditioned; the CGNE converges within the maximum number of iterations, but in Table \ref{tab:iteration_counts_group2_case_c} we note a drastic improvement by preconditioning it. This is reasonable since we observe that the singular values of both problems are uniformly bounded away from zero. This gives the minimal Frobenius preconditioner ideal properties similar to the one described in \cite{compression1,compression2}.
		
	The modified group 3 is positive definite and symmetric, so the CG behaves well and is predictable. Empirically, we see that the singular values of the preconditioned matrix do not decay significantly, and the larger singular values do not grow. The latter explains the fact that the number of iterations seems to remain constant, and we expect a superlinear convergence of the method, see Table \ref{tab:iteration_counts_group3}.

	In all experiments, the CG we applied has good and reasonable behavior. The {convergence} behavior of GMRES is theoretically more complex, but we see {again} good properties of the numerical results.
	
	\begin{table}[!h]
		\centering
		\begin{tabular}{|c|c|c|c|c|c|c|c|c|c|}
			\toprule
		Case	& $\eta$ & $d_n$ & iter. & Norm res. & iter. & Norm res. & $\lambda$ \\
			& & & GMRES & GMRES &  PGMRES & PGMRES &  \\
			\midrule
        & 100 & 500  & 1000 & 2.72e-02 & 23 & 1.47e-07 & 0.046 \\
        (a) & 200 & 1000 & 1000 & 3.43e-02 & 27 & 1.68e-07 & 0.027 \\
        & 500 & 2500 & 1000 & 3.14e-02 & 33 & 3.44e-07 & 0.013 \\
        \midrule
       & 100 & 520  & 1000 & 2.78e-02 & 21 & 1.82e-07 & 0.042 \\
       (b) & 200 & 1020 & 1000 & 3.43e-02 & 25 & 3.67e-07 & 0.025 \\
       & 500 & 2520 & 1000 & 3.21e-02 & 33 & 8.22e-07 & 0.013 \\
        \midrule
        & 100 & 510  & 1000 & 2.95e-02 & 22 & 1.38e-07 & 0.044 \\
        (c) & 200 & 1015 & 1000 & 2.43e-02 & 28 & 4.73e-08 & 0.028 \\
        & 500 & 2523 & 1000 & 3.21e-02 & 32 & 1.64e-07 & 0.013 \\
			\bottomrule
		\end{tabular}
		\caption{Group 1: number of iterations and norm of the residuals of (P)GMRES across various values of $\eta$, where $d_n$ is the size of the problem and $\lambda= \text{iter. PGMRES}/N$.}
		\label{tab:iteration_counts}
	\end{table}

	\begin{table}[h!]
		\centering
		\begin{tabular}{|c|c|c|c|c|c|c|}
			\toprule
			$\eta$ & $d_n$ & iter. CGNE & Res. CGNE & iter. PCGNE & Res. PCGNE & $\lambda$ \\
			\midrule
     100 & 310  & 57 & 1.59e-07 & 18 & 6.51e-08 & 0.058 \\
        200 & 615  & 57 & 2.09e-07 & 18 & 8.47e-08 & 0.029 \\
        500 & 1522 & 56 & 3.00e-07 & 18 & 8.35e-08 & 0.012 \\
			\bottomrule
		\end{tabular}
		\caption{Group 2, case (c): number of iterations and norm of the residuals of (P)CGNE across various values of $\eta$, where $d_n$ is the number of rows of $A_n$ and $\lambda= \text{iter. PCGNE}/N$.}
\label{tab:iteration_counts_group2_case_c}
	\end{table}
	
	\begin{table}[!h]
		\centering
		\begin{tabular}{|c|c|c|c|c|c|c|}
			\toprule
			$\eta$ & $d_n$ & iter. CG & Res. CG & iter. PCG & Res. PCG & $\lambda$ \\
			\midrule
			100 & 696 & 226 & 6.07e-07 &   23 & 6.51e-07 &0.15 \\
			200 & 1396 & 382 & 9.06e-07 & 27 & 3.29e-07 & 0.11 \\
			500 & 3496 & 496 & 1.44e-06 & 28 & 4.95e-07 & 0.07 \\
			\bottomrule
		\end{tabular}
		\caption{Group 3: number of iterations and norm of the residuals of (P)CG across various values of~$\eta$, where $d_n$ is the size of the problem and $\lambda= \text{iter. PCG}/N$.}
	\label{tab:iteration_counts_group3}
	\end{table}
		
	\begin{table}[!h]
		\centering
		\resizebox{\textwidth}{!}{
		\begin{tabular}{|c|c|c|c|c|c|c|c|c|}
			\toprule
			Case & $\eta$ & $\min \sigma(A_n)$ & $\min \sigma(M_n)$ &  $\max \sigma(A_n)$ &  $\max \sigma(M_n)$ & $\mu(A_n)$ & $\mu(M_n)$ \\
			\midrule
			& 100 & 9.53e-04 & 3.13e-02 & 8.31 & 7.69 & 8.72e+03 & 2.46e+02 \\
		(a) & 200 & 2.44e-04 & 1.67e-02 & 8.32 & 26.43 & 3.41e+04 &  1.59e+03 \\
		&	500 & 3.91e-05 & 9.10e-03 & 8.32 & 92.04  & 2.13e+05 & 1.01e+04 \\
        \midrule
        & 100 & 2.99e-04 & 2.30e-04 &  8.32 & 23.56 & 2.78e+04 & 1.03e+03 \\
	(b)	& 200 & 1.53e-03  & 7.94e-03 & 8.31 & 10.14 & 5.43e+03 &  1.28e+03 \\
	& 500 & 4.23e-05 & 7.30e-03 & 8.32 & 164.76 & 1.97e+05 & 2.26e+04 \\
    \midrule
     & 100 & 1.14e-03 & 2.23e-02 & 8.32 & 8.04 & 7.29e+03 & 3.60e+02  \\
    (c) & 200 & 2.46e-04 & 2.59e-03 & 8.32 & 26.49 & 3.38e+04 & 1.02e+04 \\
    & 500 & 4.30e-05 & 1.01e-02 & 8.32 & 45.94 & 1.93e+05 & 1.15e+04 \\
			\bottomrule
			\end{tabular}
		}
		\caption{Group 1: extreme singular values and conditioning. Here $M_n=S_n^{-1}A_n$.}
		\label{tab:singular_values_1a}
	\end{table}
	
	\begin{table}[!h]
		\centering
		\resizebox{\textwidth}{!}{
		\begin{tabular}{|c|c|c|c|c|c|c|c|c|}
			\toprule
			Case & $\eta$ & $\min \sigma(A_n)$ & $\min \sigma(M_n)$ &  $\max \sigma(A_n)$ &  $\max \sigma(M_n)$ & $\mu(A_n)$ & $\mu(M_n)$ \\
			\midrule
			& 100 & 1.610 & 0.36 &  11.74 & 1.62 & 7.28e+00 & 4.50e+00 \\
		(a) &	200 & 1.607 & 0.36   & 11.74 & 1.62 &  7.31e+00 & 4.51e+00 \\
		&	500 & 1.605 & 0.36  & 11.75 & 1.62 & 7.32e+00 & 4.53e+00 \\
        \midrule
    & 100 & 1.610 & 0.36 & 11.74 & 1.62   & 7.28e+00 & 4.50e+00 \\
	(b) &	200 & 1.607 & 0.36   & 11.74 & 1.62 & 7.31e+00 & 4.51e+00 \\
	& 500 & 1.605 & 0.36 & 11.75 & 1.62 & 7.32e+00 & 4.53e+00 \\
    \midrule
    & 100 & 1.610 & 0.34 & 11.74 & 1.62  & 7.28e+00 & 4.51e+00 \\
    (c) &	200 & 1.606 & 0.36  & 11.74 & 1.62 & 7.31e+00 & 4.51e+00 \\
    & 500 & 1.604 & 0.36 & 11.75 & 1.62 & 7.32e+00 & 4.53e+00 \\
			\bottomrule
			\end{tabular}}
		\caption{Group 2: extreme singular values and conditioning. Here $M_n=A_nS_n^{\dag}$.}
		\label{tab:singular_values_2a}
	\end{table}

	\begin{table}[!h]
		\centering
		\resizebox{\textwidth}{!}{
		\begin{tabular}{|c|c|c|c|c|c|c|c|}
			\toprule
			$\eta$ & $\min \sigma(A_n)$ & $\min \sigma(M_n)$ &  $\max \sigma(A_n)$ &  $\max \sigma(M_n)$ & $\mu(A_n)$ & $\mu(M_n)$ \\
			\midrule
			100& 1.47e-3 & 3.65e-2 & 11.62 & 17.22 & 7.90e+03 & 4.72e+02 \\
			200& 7.67e-4 & 3.65e-2 & 11.62 & 26.78 & 1.52e+04 & 7.33e+02 \\
			500& 5.62e-4 & 3.81e-2 & 11.61& 45.90 & 2.07e+04 & 1.21e+03 \\
			\bottomrule
			\end{tabular}}
		\caption{Group 3: extreme singular values and conditioning. Here $M_n=S_n^{-1}A_n$.}
		\label{tab:singular_values_3mod}
	\end{table}
	
For the (P)CG method we have a theoretical prediction result of the speed of the algorithm when we apply the preconditioner \cite{AxL} (see also \cite{Krylov-book} for similar results for other (preconditioned) Krylov methods).
\newpage
	\begin{theorem}\label{serra_outliers}
		Consider the function
		\begin{displaymath}
			k^*(a, b, \varepsilon) = \left\lceil \frac{\log 2 \varepsilon^{-1}}{\log \sigma} \right\rceil
		\end{displaymath}
		with $\sigma = \frac{\sqrt{b} - \sqrt{a}}{\sqrt{b} + \sqrt{a}}$.
		Suppose the spectrum $\Sigma_n$ of $P_n^{-1} A_n$ follows one of the three cases:
		\begin{enumerate}
			\item $\Sigma_n \subset [a, b]$ with at most $q$ outliers bigger than $b$;
			\item $\Sigma_n \subset [a, b]$ with at most $q$ outliers lower than $a$;
			\item $\Sigma_n \subset [a, b]$ with at most $2q$ outliers, $q$ bigger than $b$ and $q$ outliers lower than $a$.
		\end{enumerate}
 		{Then, to reach} a fixed accuracy $\varepsilon$ with the PCG method with preconditioner $P_n$, $N(\varepsilon)$ iterations are sufficient according to the cases below:
		\begin{enumerate}
			\item $N(\varepsilon) = q + k^*(a, b, \varepsilon)$;
			\item $N(\varepsilon) = q + k^*(a, b, \varepsilon^*)$, with $\varepsilon^* = \varepsilon \cdot O(\mu(n))$;
			\item $N(\varepsilon) = q + k^*(a, b, \varepsilon^*)$, with $\varepsilon^* = \varepsilon \cdot O(\mu(n))$.
		\end{enumerate}
		Here $\mu(n)$ denotes the spectral condition number of $P_n^{-1} A_n$ that is the standard Euclidean condition number of its symmetrized version
		$P_n^{-1/2} A_nP_n^{-1/2}$.
	\end{theorem}
Given the Theorem \ref{serra_outliers}, if we consider the empirical data of the spectrum of matrices and the numerical results of the Krylov method used, given the uncertainty of the cluster at $1$ set at $0.5$, the convergence tolerance set to $10^{-8}$ and the hyperparameter we set for the experiment, we could see {empirically} that PCG and PGMRES show a number of iterations of the form
	\begin{equation}\label{estimate}
		q-6.072\log\left(\frac{2}{\mu(n)}\right)+\hat{C}_{\eta},
	\end{equation}
	where $q$ is the number of outliers greater than $1$ and $\hat{C}_{\eta}$ is an asymptotically constant number (in our test between 30 and 40), while the condition number $\mu(n)$ is reported in tables \ref{tab:singular_values_1a}--\ref{tab:singular_values_3mod}.
    We expect this behavior to be maintained if we consider larger and larger matrices.
	

    Unfortunately, Theorem \ref{serra_outliers} does not provide an estimation of $\hat{C}_{\eta}$. Nevertheless, in our tests, PCG and PGMRES converge in a reasonable number of iterations, although Theorem \ref{serra_outliers} does not predict the related behavior.

\subsection{Numerical results from fractional diffusion equations}\label{ssec:fract}
In this section, by exploiting the findings in Section \ref{sec:Toeplitz_block} and Section \ref{sec:appr-prec}, we test further preconditioners for dense matrices of the form (\ref{eq:A_Toeplitz}), both in the unilevel and in the two-level setting.
According to (\ref{eq:A_Toeplitz}), in the unilevel setting, the matrix sizes are proportional to $n$, while in the two-level setting each block in (\ref{eq:A_Toeplitz}) is a two-level Toeplitz structure. Hence, the resulting matrix-sizes are proportional to $n^2$.
The preconditioners follow the structure described in (\ref{eq:hat A_Toeplitz}), with each nonzero block being replaced by a circulant counterpart.
Of course, the choice is computationally convenient since the overall inversion of the preconditioner can be done within $O(n^d\log n)$ arithmetic operations, where $d=1$ in the unilevel case and $d=2$ in the two-level case. For specific cases better selections can be considered, as the use of the $\tau$ algebra, when all the involved generating functions are even: in this respect, the reader is referred to \cite{linear} for a theoretical analysis motivating the superior performances of the $\tau$ algebra approximation with respect to the circulant approximation, and to \cite{fract1,fract2} for specific studies, when fractional differential equations are considered.

In Example \ref{exl:ex1_1d_positivePrec_sparseResidual} and Example \ref{exl:ex2_1d_nonPositivePrec_denseResidual} we consider the matrix-sequence $\{A_n\}_n$, where, for each $n$, $A_n$ is a matrix as in (\ref{eq:A_Toeplitz}) with $s=t=1$ and $\nu=2$. More in detail, for the unilevel examples, we choose
$n_2=2n$, $n_1=n$ so that the global matrix-size is $d_n=3n$. Furthermore, we set
\begin{displaymath}
	A_n=\begin{bmatrix}
		T_{n_1}(f_{1,1}) &  T_{n_1,n_2}(f_{1,2})\\
		T_{n_2,n_1}(f_{2,1}) &  T_{n_2}(f_{2,2})
	\end{bmatrix}
\end{displaymath}
and the preconditioner
\begin{displaymath}
	P_n=\begin{bmatrix}
		C_{n1}(T_{n_1}(f_{1,1})) &  C_{n1}(T_{n_1}(f_{1,2})) & \mathbf{0}_{n1}          \\
		C_{n1}(T_{n_1}(f_{2,1})) &  C_{n1}(T_{n_1}(f_{2,2})) & \mathbf{0}_{n1}          \\
		\mathbf{0}_{n1}          &  \mathbf{0}_{n1}          & C_{n1}(T_{n_1}(f_{2,2}))
	\end{bmatrix}
\end{displaymath}
whith $C_n(T_n(f))$ being the Frobenius optimal preconditioner for $T_n(f)$ \cite{T-Chan,compression1,compression2}. The preconditioner is block circulant so all operations can be performed within $O(n\log n)$ cost. {The zero blocks of the preconditioner form a special matrix-sequence which correspond to a \emph{zero-distributed} residual of $\{A_n\}_n$ after subtracting the part that essentially determines the spectral behavior as discussed in Theorem \ref{th:general}; see (\ref{eq:hat A_Toeplitz}).}

In the two-level setting, we choose the same proportions between $n_1$ and $n_2$. However, each block is a two-level Toeplitz structure constructed using tensor products so that the whole matrix has size $d_n=3n^2$.

A further feature that makes all the considered examples challenging is the fact that we use as diagonal blocks approximations of fractional differential operators so that the whole structures are dense with slowly decaying coefficients.

In the specific examples, we report the generating functions that we use and report the numerical tests and the visualizations, in connection with the theoretical findings.

\begin{example}\label{exl:ex1_1d_positivePrec_sparseResidual}
	Here we employ
	\begin{align*}
			f_{1,1}(\theta)&=-{\rm e}^{-i\theta}(1-{\rm e}^{i\theta})^{\alpha}-{\rm e}^{i\theta}(1-{\rm e}^{-i\theta})^{\alpha}, \qquad
			f_{1,2}(\theta)=-1+{\rm e}^{-i\theta},\\
			f_{2,1}(\theta)&=f_{1,2}(-\theta),
            \qquad
			f_{2,2}(\theta)=4-2\cos(\theta),
		\end{align*}
the function $f_{1,1}$ stemming from the discretization of a Fractional Diffusion Equation {with constant coefficients equal to $1$} using a first-order finite difference scheme, with fractional exponent $\alpha$ set to $1.5$.
To be precise, in \cite{fract1} it was shown that
\begin{align*}
f_{1,1}(\theta)= -2g_{1}^{(\alpha)}-(g_{0}^{(\alpha)}+g_{2}^{(\alpha)})({\rm e}^{-i \theta}+{\rm e}^{i \theta})-\sum_{k=2}^{\infty}g_{k+1}^{(\alpha)}
({\rm e}^{-k i \theta}+{\rm e}^{i k \theta}),
\end{align*}
 with
\begin{align}
 g_k^{(\alpha)}=(-1)^k\binom{\alpha}{k}=\frac{(-1)^k}{k!}\alpha(\alpha-1)\cdots(\alpha-k+1),\quad k= 0,1,\ldots. \label{defin_g_aplha_coeff}
\end{align}
The function is real and positive with a zero of order $\alpha$ at $0$, so the $T_{n1}(f_{1,1})$ is an ill-conditioned, dense, and positive definite real symmetric matrix with its minimum eigenvalue $\lambda_{\min}(T_n(f_{1,1}))\sim O(n^{-\alpha})$. In this example, both the main matrix $A_n$ and the preconditioner $P_n$ are real, symmetric, and positive definite. Thus the eigenvalues of $A_n$ and the eigenvalues of the preconditioned matrix $P_n^{-1}A_n$ are positive.

Figure \ref{fig:ex1_eigs_sym_a15_n100_200_400} shows that the eigenvalues of the preconditioned matrix are clustered at $1$, while the eigenvalues of the non-preconditioned matrix are scattered uniformly between the minimum of the minimal eigenvalue of the spectral symbol and the maximum of the maximal eigenvalue of the spectral symbol.
	
In Table \ref{table:ex1_eigs_sym_a15_n100_200_400} it is shown that $1$ is indeed a weak cluster for the preconditioned matrix-sequence $\{P_n^{-1}A_n\}_n$ as the percentage of eigenvalues lying outside the interval $[0.95, 1.05]$ approaches zero, as the matrix-size increases. In the same table, the iterations needed for solving a system with coefficient $A_n$ by employing (P)CG are presented. Again the results conform to the expectations, since the slight increase in the iteration count, probably of polylogarithmic type, can be explained by the moderate increase of the absolute number $N$ of the outliers.
	
	\begin{table}
			\centering
			\begin{tabular}{|c|c|c|c|c|c|}
					\toprule
					$n_1$ & $d_n$ & $N$   & $N/d_n$ & Iter. CG  & Iter. PCG \\
					\midrule
					$100$ & $300$ & $22$  & $0.073$ &   $>100$    & $15$  \\
					$200$ & $600$ & $28$  & $0.046$ &   $>100$    & $16$  \\
					$400$ & $1200$& $37$  & $0.030$ &   $>100$    & $18$  \\
					\bottomrule
				\end{tabular}
            \caption{Example \ref{exl:ex1_1d_positivePrec_sparseResidual}: For increasing matrix size $d_n\in\{300,600,1200\}$ the absolute $N$, and the relative $N/d_n$ number of the eigenvalues of $P_n^{-1}A_n$ lying outside the fixed interval $[0.95,1.05]$ are presented. The number of the iterations needed by the CG and PCG methods is also presented.}\label{table:ex1_eigs_sym_a15_n100_200_400}
	\end{table}
	
	\begin{figure}
			\centering		\includegraphics[width=.48\textwidth]{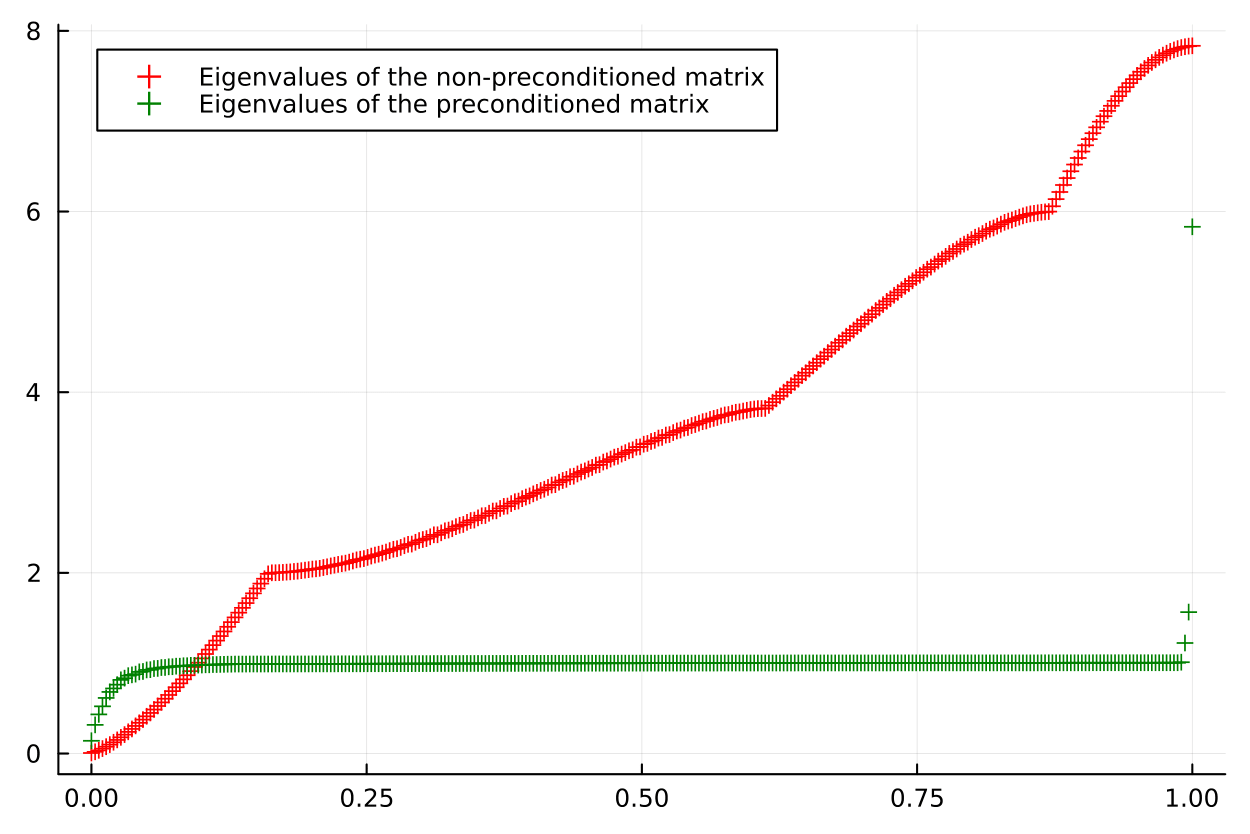}	
            \hspace{0.2cm}
            \includegraphics[width=.48\textwidth]{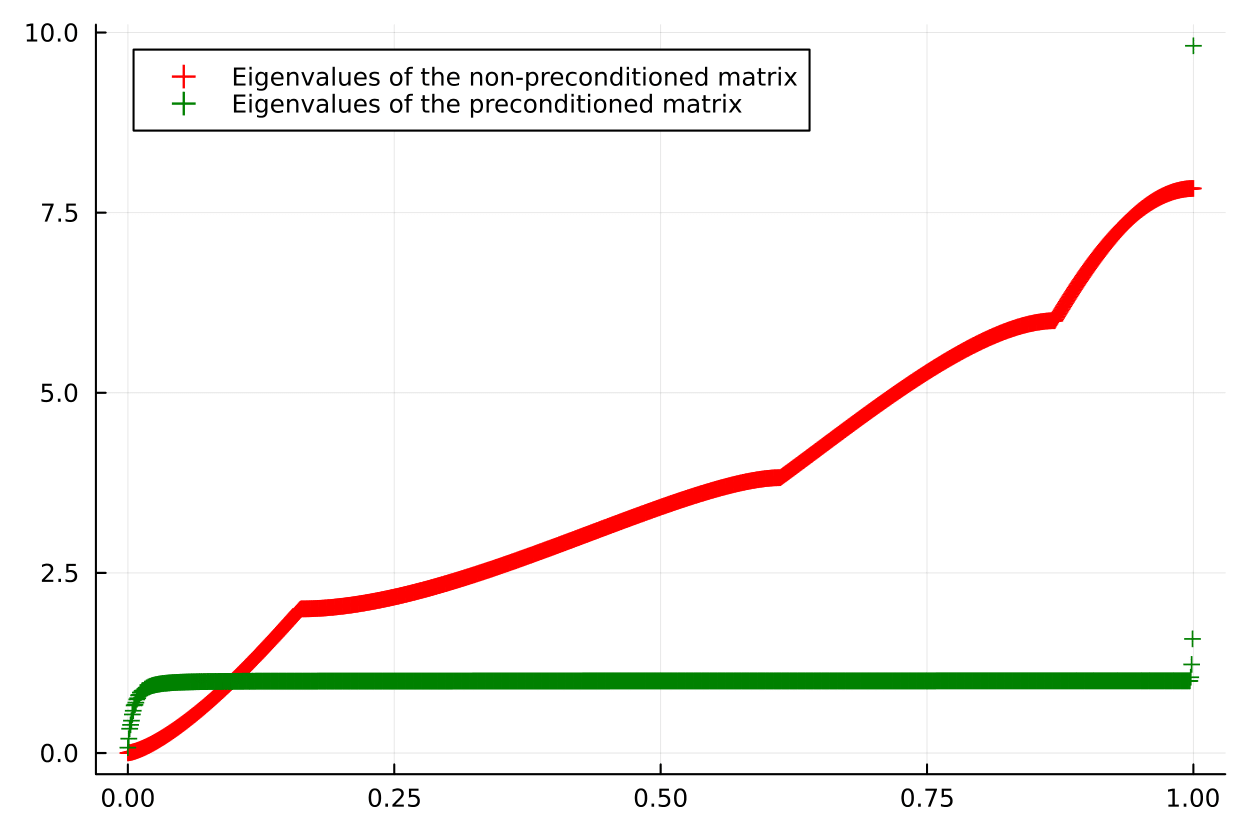}
			\caption{Example \ref{exl:ex1_1d_positivePrec_sparseResidual}: The eigenvalues of the non-preconditioned against the eigenvalues of the preconditioned matrix, for $n_1=100$ (left) and $n_1=400$ (right), arranged in non-decreasing order, where $n_2=2n_1$.}
            \label{fig:ex1_eigs_sym_a15_n100_200_400}
		\end{figure}
\end{example}

\begin{example}\label{exl:ex2_1d_nonPositivePrec_denseResidual}
	In this example we set
	\begin{align*}
			f_{1,1}(\theta)&=2-2\cos(\theta),\qquad
			f_{1,2}(\theta)=-{\rm e}^{-i\theta}(1-{\rm e}^{i\theta})^{\alpha}-{\rm e}^{i\theta}(1-{\rm e}^{-i\theta})^{\alpha}, \\
			f_{2,1}(\theta)&=f_{1,2}(-\theta),\qquad
			f_{2,2}(\theta)=4-2\cos(\theta),
		\end{align*}
	with $\alpha=1.7$. In contrast with the first case, both the coefficient matrix $A_n$ and the preconditioner have negative eigenvalues. As a consequence, the preconditioned matrix has complex eigenvalues. However, the eigenvalues of the preconditioned matrix still have a weak cluster at $1$ as evident in Figure \ref{fig:ex2_eigs_nonSym_a17_n100_200_400} and verified by the data in Table \ref{table:ex2_eigs_nonSym_a17_n100_200_400}. In the same table, the iterations needed for solving a system with coefficient matrix {$A_n$} - employing the  GMRES method and the PGMRES method - are also presented.
	\begin{table}
			\centering
			\begin{tabular}{|c|c|c|c|c|c|}
					\toprule
					$n_1$ & $d_n$ & $N$   & $N/d_n$ & Iter. GMRES  & Iter. PGMRES \\
					\midrule
					$100$ & $300$ & $31$  & $0.103$ &   $>100$    & $18$  \\
					$200$ & $600$ & $38$  & $0.063$ &   $>100$    & $21$  \\
					$400$ & $1200$& $50$  & $0.041$ &   $>100$    & $26$  \\
					\bottomrule
				\end{tabular}
                \caption{Example \ref{exl:ex2_1d_nonPositivePrec_denseResidual}: For increasing matrix size $d_n\in\{300,600,1200\}$ the absolute $N$, and the relative $N/d_n$ number of the eigenvalues of $P_n^{-1}A_n$ lying outside the fixed interval $[0.95,1.05]$ are presented.
                The number of the iterations needed by the GMRES and PGMRES methods is also presented. }\label{table:ex2_eigs_nonSym_a17_n100_200_400}
		\end{table}
	
	\begin{figure}
			\centering
			\includegraphics[width=.48\textwidth]{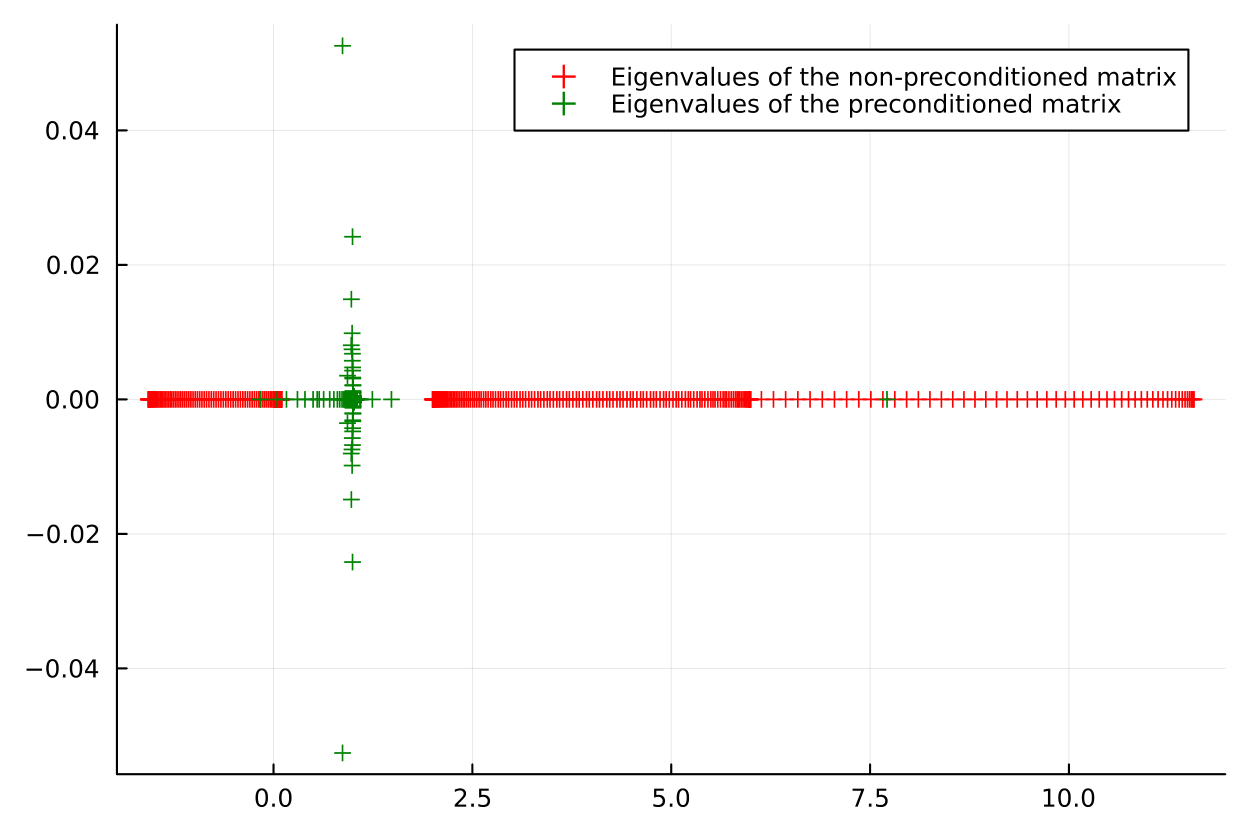}\hspace{0.2cm}
		\includegraphics[width=.48\textwidth]{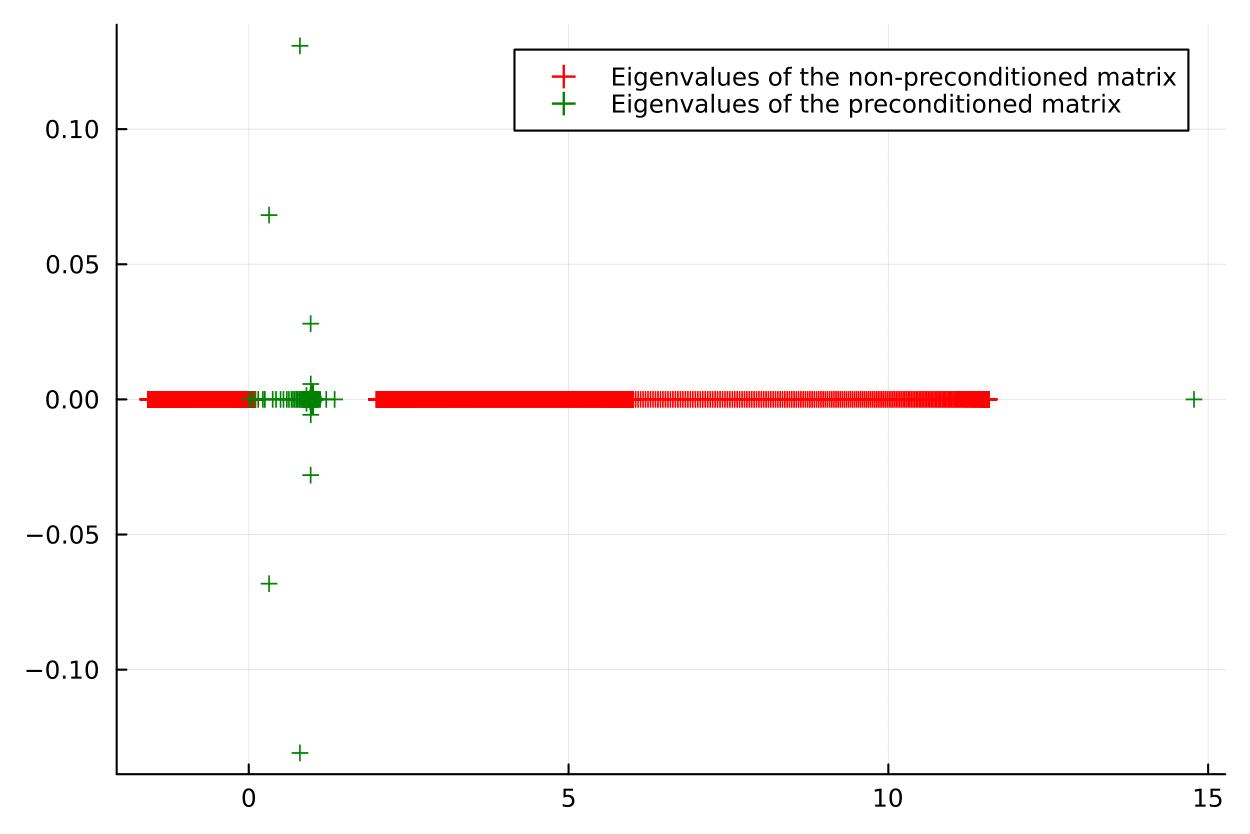}
		\caption{Example \ref{exl:ex2_1d_nonPositivePrec_denseResidual}: The eigenvalues of the non-preconditioned against the eigenvalues of the preconditioned matrix, for $n_1=100$ (left) and $n_1=400$ (right), arranged in non-decreasing order, where $n_2=2n_1$.}\label{fig:ex2_eigs_nonSym_a17_n100_200_400}
		\end{figure}
\end{example}

\begin{example}\label{ex3:ex1_2d_PositivePrec_denseResidual}
{
	In this example, we explore the findings of the previous sections in a two-level setting and each block of the main structure $A_n$ comes from the discretization of a fractional differential operator in dimension $2$. Specifically
	\begin{displaymath}
	A_n=\begin{bmatrix}
		A_{1,1} &  A_{1,2} \\
		A_{2,1} &  A_{2,2}
	\end{bmatrix},
    \end{displaymath}
with $A_{1,1}$ and $A_{2,2}$ stemming from the descritization of a two-level fractional differential equation with constant coefficients using a second-order finite difference scheme. More in detail
\begin{align*}
A_{1,1}&= 4( {I}_n \otimes T_n(q_{\alpha}(\theta_1)) + T_n(q_\beta(\theta_2)) \otimes  {I}_n  ),\\
A_{2,2}&=    {I}_{2n} \otimes T_n(q_{\alpha}(\theta_1)) + T_{2n}(q_\beta(\theta_2)) \otimes  {I}_n,
\end{align*}
with $q_{\gamma}(\theta)=w_{\gamma}(\theta)+w_{\gamma}(-\theta)$, $w_\gamma(\theta)=-\left(\frac{2-\gamma(1-{\rm e}^{-i\theta})}{2}\right)\left(1-{\rm e}^{i\theta}\right)^\gamma$ and $q_{\gamma}(\theta)$ owing the Fourier expansion

\begin{align*}
q_{\gamma}(\theta)&= -2w_{1}^{(\gamma)}-(w_{0}^{(\gamma)}+w_{2}^{(\gamma)})({\rm e}^{-i \theta}+{\rm e}^{i \theta})-\sum_{k=2}^{\infty}w_{k+1}^{(\gamma)}({\rm e}^{-k i \theta}+{\rm e}^{k i \theta}),\\
w_0^{(\gamma)}&=\frac{\gamma}{2}g_0^{(\gamma)},\quad   w_k^{(\gamma)}=\frac{\gamma}{2}g_k^{(\gamma)}+\frac{2-\gamma}{2}g_{k-1}^{(\gamma)}, \quad k\geq 1,
\end{align*}
and $g_k^{(\gamma)}$ as in \eqref{defin_g_aplha_coeff}. Thus the two diagonal blocks of $A_n$ are the two-level Toeplitz matrices
with generating functions $f_{i,i}(\theta_1,\theta_2)= c_i( q_{\alpha}(\theta_1)+q_{\beta}(\theta_2))$ for $i=1,2$ and $c_1=4,\:c_2=1$ and discretization orders $(n,n)$ and $(n,2n)$, respectively. We mention that these functions are real positive and even with a zero of order $\min(\alpha,\beta)$ at zero, and the matrices are real symmetric and ill-conditioned with their minimum eigenvalue
$\lambda_{min}T_{(n_1,n_2)}^{(2)}(f_{i,i})\sim O(n^{-\min(\alpha,\beta)})$ (see \cite{moghaderi171} for a detailed analysis on the matter). For both submatrices was set $\alpha=1.5$ and $\beta=1.8$. Finally $A_{1,2}$ was set to be the $n^2\times 2n^2$ rectangular upper part of a two-level Laplacian
matrix, while $A_{2,1}=A_{1,2}^T$. The assembled $A_n\in \mathbb{R}^{3n^2\times3n^2}$ is real, symmetric, and positive definite.
\newline
The preconditioner was set to be
\begin{displaymath}
	P_n=\begin{bmatrix}
		\tau_n(f_{1,1}) &  \tau_n(f_{2,1}) & \mathbf{0}_{n}          \\
		\tau_n(f_{2,1}) &  \tau_n(f_{2,2}) & \mathbf{0}_{n}          \\
		\mathbf{0}_{n} &  \mathbf{0}_{n} & \tau_n(f_{2,2})
	\end{bmatrix},
\end{displaymath}
with the nonzero blocks selected from the $\tau$ algebra. In \cite{fract2} it was shown that $\tau$ preconditioners are effective for two-dimensional Toeplitz-like systems whenever the generating function is real positive and even with zeros of noninteger order. More precisely,
\begin{align*}
\tau_n(f(\theta_1,\theta_2)) = \left(\mathbb{S}_{n} \otimes \mathbb{S}_{n}\right)D\left(\mathbb{S}_{n} \otimes \mathbb{S}_{n}\right),
\end{align*}
$D$ being the diagonal matrix with the sampling of $f$ over an $n\times n$ grid of its domain and $\mathbb{S}_{n}$ the discrete sine transform matrix defined by
\begin{align*}
[\mathbb{S}_n]_{i,j}&=\sqrt{\frac{2}{n+1}}\sin{\left(i\theta_j\right)},  \qquad i,j=1,\ldots, n.
\end{align*}
Since the Laplacian belongs to the $\tau$ algebra the $\tau_n(f_{2,1})$ is the Laplacian itself. All operations can be performed in $O(n\log n)$ cost. The preconditioner is real symmetric and positive definite as the main matrix and all eigenvalues are real.
As in the previous examples, Figure \ref{fig:ex3_2d_eigs_nonSym_a11_b18_n1200_2700_4800} and Table \ref{table:ex3_2d_eigs_nonSym_a11_b18_n1200_2700_4800} shown that $1$ is a weak cluster for the preconditioned matrix-sequence $\{P_n^{-1}A_n\}_n$ as the percentage of eigenvalues lying outside the interval $[0.95, 1.05]$ approaches zero, as the matrix-size increases.
}

	\begin{table}
			\centering
			\begin{tabular}{|c|c|c|c|c|c|}
					\toprule
					$n$ & $d_n$ & $N$   & $N/d_n$ & Iterations  & Iterations \\
					    &       &       &         & for CG      & for PCG \\
					\midrule
					$20$ & $1200$ & $106$  & $0.088$ &   $>120$    & $14$   \\
					$30$ & $2700$ & $162$  & $0.060$ &   $>169$    & $16$   \\
					$40$ & $4800$ & $217$  & $0.045$ &   $>215$    & $18$   \\
					\bottomrule
			\end{tabular}
            			\caption{ Example \ref{ex3:ex1_2d_PositivePrec_denseResidual}: For increasing matrix-size $d_n\in\{1200,2700,4800\}$ the absolute $N$, and the relative $N/d_n$ number of the eigenvalues of $P_n^{-1}A_n$ lying outside the interval $[0.95,1.05]$ are presented. The number of the iterations of CG and PCG methods is also presented. }\label{table:ex3_2d_eigs_nonSym_a11_b18_n1200_2700_4800}
		\end{table}
	
	\begin{figure}
			\centering
			\includegraphics[width=.48\textwidth]{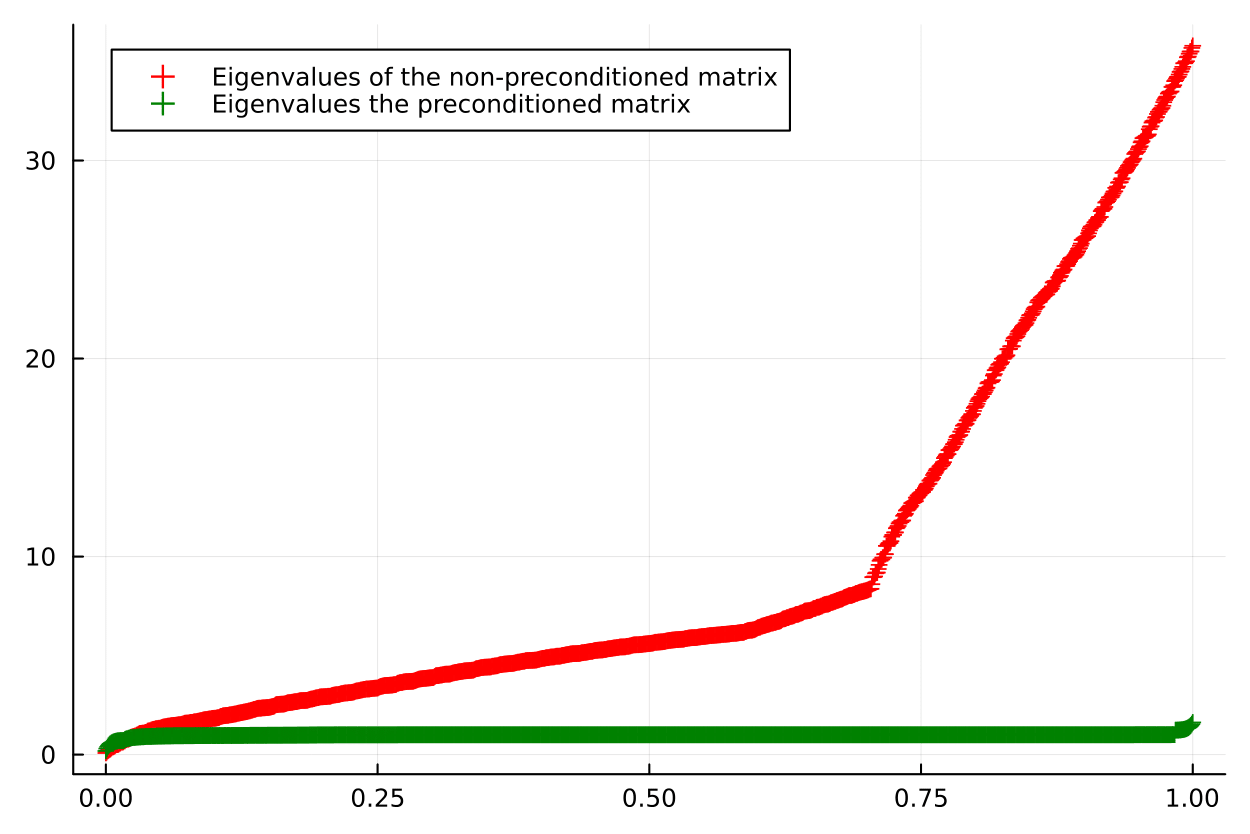}
            \hspace{0.2cm}
			\includegraphics[width=.48\textwidth]{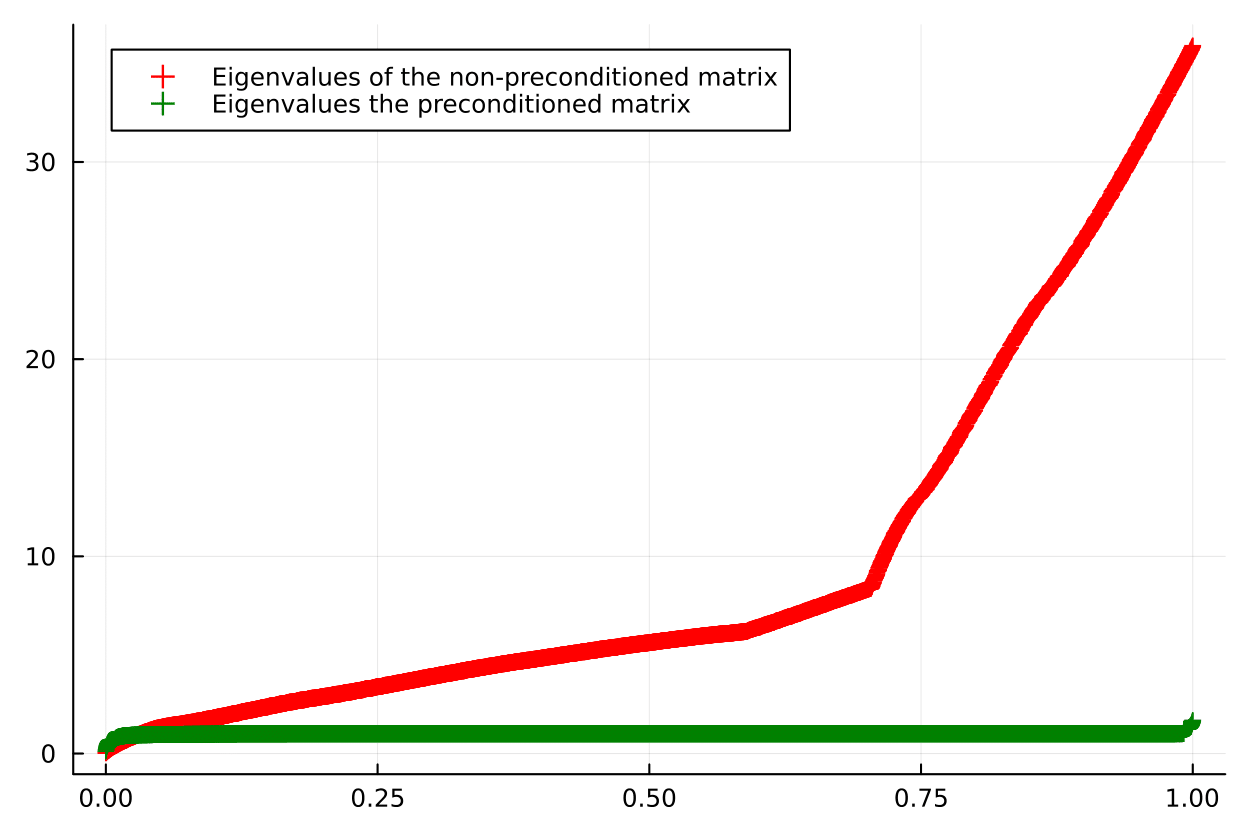}
			\caption{
            Example \ref{ex3:ex1_2d_PositivePrec_denseResidual}: The eigenvalues of the non-preconditioned against the eigenvalues of the preconditioned matrix in the two-dimensional case, for matrix sixe $n=1200$ (left) and $n=4800$ (right), arranged in non-decreasing order. }\label{fig:ex3_2d_eigs_nonSym_a11_b18_n1200_2700_4800}
		\end{figure}
\end{example}

\section{Conclusions}\label{sec:fin}
We have provided a general framework for designing efficient preconditioning strategies for large linear systems whose coefficient matrices are extracted from matrix-sequences in block form. We have dealt in detail the case where the matrix-sequences associated with the blocks are of block unilevel Toeplitz type, proving clustering results for the resulting preconditioned matrix-sequences. Several numerical tests have been presented and critically discussed.

When thinking of future steps, the primary open challenge lies in generalizing the results to the multilevel setting, enabling the treatment of approximations for multidimensional differential or fractional differential problems with variable coefficients: in this direction, the numerical experiments in Section \ref{ssec:fract} for two-level dense structures provide a heuristic ground for the related multilevel analysis.
Such a generalization can follow various possible directions and has to include the case where the blocks are of block multilevel Toeplitz type  or derived from block multilevel GLT matrix-sequences, taking into account the theoretical barriers in \cite{Nega-ill,Nega-gen} for multilevel structures:
in such a case, a way for overcoming the difficulties is the multi-iterative strategy (see \cite{multi} for the original proposal), which consists of combining stationary/(preconditioned) Krylov iterative solvers and multigrid methods (see e.g. \cite{dumb,EMI} and references therein). This will involve notational and technical challenges that need to be handled with care and precision.

\section*{Acknowledgements}
	Marco Donatelli, Samuele Ferri, Valerio Loi, Stefano Serra-Capizzano, Rosita L. Sormani are partly supported by ``Gruppo Nazionale per il Calcolo Scientifico" (INdAM-GNCS). 	
    Marco Donatelli is partially supported by the PRIN project 2022ANC8HL funded by the Italian Ministry of University and Research.
    The work of Stefano Serra-Capizzano is funded by the European High-Performance Computing Joint Undertaking  (JU) under grant agreement No 955701. The JU receives support from the European Union’s Horizon 2020 research and innovation program and Belgium, France, Germany, and Switzerland. Furthermore, Stefano Serra-Capizzano is grateful for the support of the Laboratory of Theory, Economics and Systems – Department of Computer Science at Athens University of Economics and Business and is grateful to
the "Como Lake center for AstroPhysics” of Insubria University.


\begin{thebibliography}{99}



\bibitem{conj-blo-I}
A. Adriani, A.J.A. Schiavoni-Piazza, S. Serra-Capizzano. Block structures, g.a.c.s. approximation, and distributions. {\em Bol. Soc. Mat. Mex. - special Volume in Memory of Prof. Nikolai Vasilevski}, to appear, 2025.

\bibitem{gacs}
A. Adriani, A.J.A. Schiavoni-Piazza, S. Serra-Capizzano, C. Tablino-Possio. Revisiting the notion of approximating class of sequences for handling approximated PDEs on moving or unbounded domains. {\em Electron. Trans. Numer. Anal.}, to appear, 2025.

\bibitem{AxL}
O. Axelsson, G. Lindskog. On the rate of convergence of the preconditioned conjugate gradient method. {\em Numer. Math.} 48:499--523, 1986.

\bibitem{fract2}
N. Barakitis, S.-E. Ekström, P. Vassalos. Preconditioners for fractional diffusion equations based on the spectral symbol. {\em Numer. Linear Algebra Appl.} 29(5):e2441, 2022.

\bibitem{pre-prequel}
N. Barakitis, P. Ferrari, I. Furci, S. Serra-Capizzano. {\em An extradimensional approach for distributional results: the case of $2\times 2$ block Toeplitz structures}. Springer Proceedings on Mathematics and Statistics, in press, 2025.




\bibitem{MR4449208}
G. Barbarino, C. Garoni, M. Mazza, S. Serra-Capizzano. Rectangular GLT sequences. {\em Electron. Trans. Numer. Anal.} 55:585--617, 2022.

\bibitem{GLT-blocks-1-dim}
G. Barbarino, C. Garoni, S. Serra-Capizzano. Block generalized locally Toeplitz sequences: theory and applications in the unidimensional case. {\em Electron. Trans. Numer. Anal.} 53:28--112, 2020.

\bibitem{GLT-blocks-d-dim}
G. Barbarino, C. Garoni, S. Serra-Capizzano. Block generalized locally Toeplitz sequences: theory and applications in the multidimensional case. {\em Electron. Trans. Numer. Anal.} 53:113--216, 2020.

\bibitem{EMI}
P. Benedusi, P. Ferrari, M.E. Rognes, S. Serra-Capizzano. Modeling excitable cells with the EMI equations: spectral analysis and iterative solution strategy. {\em J. Sci. Comput.} 98:58, 2024.

\bibitem{Bhatia-book}
R. Bhatia. {\em Matrix Analysis}. Graduate Texts in Mathematics, Springer-Verlag, New York, 1997.

\bibitem{mixed1}
D. Boffi, L. Gastaldi, M. Ruggeri. Mixed formulation for interface problems with distributed Lagrange multiplier. {\em Comput. Math. Appl.} 68:2151-2166, 2014.

\bibitem{MR4389580}
M. Bolten, M. Donatelli, P. Ferrari, I. Furci. A symbol-based analysis for multigrid methods for block-circulant and block-Toeplitz systems. {\em SIAM J. Matrix Anal. Appl.} 43(1):405--438, 2022.

\bibitem{MR4623368}
M. Bolten, M. Donatelli, P. Ferrari, I. Furci. Symbol based convergence analysis in block multigrid methods with applications for Stokes problems. {\em Appl. Numer. Math.} 193:109--130, 2023.



\bibitem{BS}
A. B\"ottcher, B. Silbermann. {\em Introduction to large truncated {T}oeplitz matrices}. Universitext, Springer-Verlag, New York, 1999.

\bibitem{braess}
D. Braess. {\em Finite elements}. Cambridge University Press, Cambridge, 2007.

\bibitem{missing-data1}
J.F. Cai, R.H. Chan, L. Shen, Z. Shen. Convergence analysis of tight framelet approach for missing data recovery. {\em Adv. Comput. Math.} 31:87--113, 2009.


\bibitem{blo vs imag2}
J.F. Cai, R.H. Chan, Z. Shen. A framelet-based image inpainting algorithm. {\em Appl. Comput. Harmon. Anal.} 24:131--149, 2008.

\bibitem{CN}
R.H. Chan, M.K. Ng. Conjugate gradient methods for Toeplitz systems. {\em SIAM Rev.} 38(3):427--482, 1996.

\bibitem{T-Chan}
T.F. Chan. An optimal Circulant preconditioner for Toeplitz systems. {\em SIAM J. Sci. Statist. Comput.} 9(4):766--771, 1988.

\bibitem{missing-data2}
V. Del Prete, F. Di Benedetto, M. Donatelli, S. Serra-Capizzano. Symbol approach in a signal-restoration problem involving block Toeplitz matrices. {\em J. Comput. Appl. Math.} 272:399--416, 2014.


\bibitem{compression2}
F. Di Benedetto, S. Serra-Capizzano. Optimal multilevel matrix algebra operators. {\em Linear Multilin. Algebra} 48(1):35--66, 2000.

\bibitem{MR3689933}
M. Donatelli, A. Dorostkar, M. Mazza, M. Neytcheva, S. Serra-Capizzano. Function-based block multigrid strategy for a two-dimensional linear elasticity-type problem. {\em Comput. Math. Appl.} 74(5):1015--1028, 2017.

\bibitem{MR4284081}
M. Donatelli, P. Ferrari, I. Furci, S. Serra-Capizzano, D. Sesana. {Multigrid methods for block-Toeplitz linear systems: convergence analysis and applications}. {\em Numer. Linear Algebra Appl.} 28(4):e2356, 2021.

\bibitem{fract1}
M. Donatelli, M. Mazza, S. Serra-Capizzano. Spectral analysis and structure preserving preconditioners for fractional diffusion equations. {\em J. Comput. Phys.} 307:262--279, 2016.

\bibitem{MR3543002}
A. Dorostkar, M. Neytcheva, S. Serra-Capizzano. Spectral analysis of coupled PDEs and of their Schur complements via Generalized Locally Toeplitz sequences in 2D. {\em Comput. Methods Appl. Mech. Engrg.} 309:74--105, 2016.

\bibitem{dumb}
M. Dumbser, F. Fambri, I. Furci, M. Mazza, S. Serra-Capizzano, M. Tavelli. Staggered discontinuous Galerkin methods for the incompressible Navier-Stokes equations: spectral analysis and computational results. {\em Numer. Linear Algebra Appl.} 25(5):e2151, 2018.



\bibitem{MR1740439}
D. Fasino, P. Tilli. Spectral clustering properties of block multilevel Hankel matrices. {\em Linear Algebra Appl.} 306(1--3):155--163, 2000.

\bibitem{FRTBS}
P. Ferrari, R.I. Rahla, C. Tablino-Possio, S. Belhaj, S. Serra-Capizzano. Multigrid for $\mathbb{Q}_{k}$ Finite Element Matrices using a (block) {T}oeplitz symbol approach. {\em Mathematics} 8(1):5, 2020.


\bibitem{prequel}
I. Furci, A. Adriani, S. Serra-Capizzano. Block structured matrix-sequences and their spectral and singular value canonical distributions: a general theory. {\em arxiv.org:2409.06465}, 2024.


\bibitem{GSI}
C. Garoni, S. Serra-Capizzano. {\em Generalized Locally Toeplitz sequences: theory and applications. Vol. I}. Springer, Cham, 2017.

\bibitem{GSII}
C. Garoni, S. Serra-Capizzano. {\em Generalized Locally Toeplitz sequences: theory and applications. Vol. II}. Springer, Cham, 2018.

\bibitem{qp}
C. Garoni, S. Serra-Capizzano, D. Sesana. Spectral analysis and spectral symbol of  $d$-variate $\Bbb{Q}_p$ Lagrangian FEM stiffness matrices. {\em SIAM J. Matrix Anal. Appl.} 36:1100--1128, 2015.

\bibitem{tom}
C. Garoni, H. Speleers, S.-E. Ekstr\"om, A. Reali, S. Serra-Capizzano, T.J.R. Hughes. Symbol-based analysis of finite element and isogeometric B-spline discretizations of eigenvalue problems: exposition and review. {\em Arch. Comput. Methods Eng.} 26(5):1639--1690, 2019.

\bibitem{MR0890515}
U. Grenander, G. Szeg\"o. {\em Toeplitz forms and their applications}. Chelsea Publishing Co., New York, Second edition, 1984.



\bibitem{huckle}
T.K. Huckle. Compact Fourier analysis for designing multigrid methods. {\em SIAM J. Sci. Comput.} 31(1):644--666, 2008.

\bibitem{MR3904142}
M. Mazza, A. Ratnani, S. Serra-Capizzano. Spectral analysis and spectral symbol for the 2D curl-curl (stabilized) operator with applications to the related iterative solutions. {\em Math. Comp.} 88(317):1155--1188, 2019.

\bibitem{moghaderi171}
H. Moghaderi, M. Dehghan, M. Donatelli, M. Mazza. Spectral analysis and multigrid preconditioners for two-dimensional space-fractional diffusion equations. {\em J. Comput. Phys.} 350:992--1011, 2017.


\bibitem{multi}
S. Serra-Capizzano. Multi-iterative methods. {\em Comput. Math. Appl.} 26(4):65--87, 1993.

\bibitem{linear}
S. Serra-Capizzano. Toeplitz preconditioners constructed from linear approximation processes.
{\em SIAM J. Matrix Anal. Appl.} 20(2):446--465, 1999.

\bibitem{compression1}
S. Serra-Capizzano. A Korovkin-type theory for finite Toeplitz operators via matrix algebras. {\em Numer. Math.} 82:117--142, 1999.

\bibitem{taud2}
S. Serra-Capizzano. Spectral behavior of matrix sequences and discretized boundary value problems. {\em Linear Algebra Appl.} 337(1--3):37--78, 2001.

\bibitem{Nega-gen}
S. Serra-Capizzano. Matrix algebra preconditioners for multilevel Toeplitz matrices are not superlinear. {\em Linear Algebra Appl.} 343--344:303--319, 2002.

\bibitem{Nega-ill}
S. Serra-Capizzano, D. Noutsos, P. Vassalos. Matrix algebra preconditioners for multilevel Toeplitz systems do not insure optimal convergence rate. {\em Theoret. Comput. Sci.} 315(2--3):557--579, 2004.


\bibitem{Se-Ti}
S. Serra-Capizzano, P. Tilli. Extreme singular values and eigenvalues of non-Hermitian block Toeplitz matrices. {\em J. Comput. Appl. Math.} 108(1--2):113--130, 1999.

\bibitem{Se-Ti-LPO}
S. Serra-Capizzano, P. Tilli. On unitarily invariant norms of matrix-valued linear positive operators. {\em J. Inequal. Appl.} 7(3):309--330, 2002.

\bibitem{MR1671591}
P. Tilli. A note on the spectral distribution of Toeplitz matrices. {\em Linear Multilin. Algebra} 45(2--3):147--159, 1998.

\bibitem{TyZ}
E. Tyrtyshnikov, N. Zamarashkin. Spectra of multilevel Toeplitz matrices: advanced theory via simple matrix relationships. {Linear Algebra Appl.}270(1--3):15--27, 1998.

\bibitem{Krylov-book}
H.A. van der Vorst. {\em Iterative Krylov Methods for Large Linear Systems}. Cambridge Monographs on Applied and Computational Mathematics, Cambridge University Press, 2003.

\end{thebibliography}
\end{document}